\documentclass[12pt]{article}
\usepackage{xparse}
\usepackage[utf8]{inputenc}
\NewDocumentCommand{\tens}{t_}
 {%
  \IfBooleanTF{#1}
   {\tensop}
   {\otimes}%
 }
\NewDocumentCommand{\tensop}{m}
 {%
  \mathbin{\mathop{\otimes}\displaylimits_{#1}}%
 }

\usepackage{amsmath,amssymb,amsthm,color,framed,mathtools}
\usepackage[hyphens]{url}
\usepackage{hyperref}
\usepackage[hyphenbreaks]{breakurl}
\usepackage[margin=1in]{geometry}
\usepackage{enumerate}
\usepackage{thmtools, thm-restate}
\usepackage{mathrsfs}
\usepackage{tikz-cd}
\tikzset{Isom/.style={above,every to/.append style={edge node={node [sloped, auto=false]{$\sim$}}}}}
\usepackage[backend=biber,style=alphabetic,
citestyle=alphabetic]{biblatex}
\usetikzlibrary{decorations.pathmorphing}
\usepackage{xspace}
\xspaceaddexceptions{]\}}

\newcommand\simto{\stackrel{\textstyle\sim}{\smash{\longrightarrow}\rule{0pt}{0.4ex}}}
\renewcommand{\O}{\mathscr{O}}
\newcommand{\Z}{\mathbf{Z}}
\newcommand{\C}{\mathbf{C}}
\newcommand{\Q}{\mathbf{Q}}

\renewcommand{\P}{\mathbf{P}}
\newcommand{\J}{\mathfrak{J}}
\newcommand{\R}{\mathbf{R}}
\newcommand{\F}{\mathcal{F}}
\newcommand{\G}{\mathcal{G}}
\newcommand{\spec}{{\rm Spec}}
\newcommand{\Jac}{{\rm Jac}}

\newcommand{\sph}{{\rm Sph}}
\newcommand{\spf}{{\rm Spf}}
\newcommand{\Hom}{{\rm Hom}}

\newcommand{\colim}{\varinjlim}

\newcommand{\ol}[1]{\overline{#1}}
\newcommand{\mc}[1]{\mathcal{#1}}
\newcommand{\mf}[1]{\mathfrak{#1}}
\newcommand{\vp}{\varphi}

\newcommand{\Et}{\'etale\xspace}
\newcommand{\hlfp}{hlfp\xspace}
\newcommand{\hfp}{hfp\xspace}
\newcommand{\hetale}{h-\'etale\xspace}
\newcommand{\hfetale}{h-formally \'etale\xspace}
\newcommand{\hsmooth}{h-smooth\xspace}
\newcommand{\hfsmooth}{h-formally smooth\xspace}
\newcommand{\et}{{\rm\acute{e}t}}
\newcommand{\het}{{\rm h\mbox{-}\et}}
\newcommand{\zar}{{\rm Zar}}
\renewcommand{\H}{{\rm H}}

\newcommand{\spcite}[2]{\cite[\href{https://stacks.math.columbia.edu/tag/#1}{{#2} {#1}}]{stackp}}

\newtheorem{theorem}{Theorem}[subsection]
\renewcommand{\thetheorem}{%
  \ifnum\value{subsection}>0
    \thesubsection
  \else
    \thesection
  \fi
  .\arabic{theorem}%
}

\newtheorem{proposition}[theorem]{Proposition}
\newtheorem{lemma}[theorem]{Lemma}
\theoremstyle{definition}
\newtheorem{definition}[theorem]{Definition}
\newtheorem{remark}[theorem]{Remark}
\newtheorem{question}[theorem]{Question}
\newtheorem{examplex}[theorem]{Example}
\newenvironment{example}
  {\pushQED{\qed}\examplex}
  {\popQED\endexamplex}
\numberwithin{equation}{subsection} 
\author{Sheela Devadas\thanks{\href{mailto:sheelad.math@gmail.com}{sheelad.math@gmail.com}}}
\title{Henselian schemes in positive characteristic}
\date{}

\begin{document}
\maketitle
\begin{abstract} The global analogue of a Henselian local ring is a Henselian pair: a ring $A$ and an ideal $I$ which satisfy a condition resembling Hensel's lemma regarding lifting coprime factorizations of polynomials over $A/I$ to factorizations over $A$. The geometric counterpart is the notion of a Henselian scheme, which is an analogue of a tubular neighborhood in algebraic geometry. 

In this paper we revisit the foundations of the theory of Henselian schemes. The pathological behavior of quasi-coherent sheaves on Henselian schemes in characteristic 0 makes them poor models for an “algebraic tube" in characteristic 0. We show that such problems do not arise in positive characteristic, and establish good properties for analogues of smooth and \Et maps in the general Henselian setting.\end{abstract}
\tableofcontents

\section{Introduction}\label{sec:intro}

\subsection{Motivation}\label{sub:motiv}

In differential geometry and topology, a {\it tubular neighborhood} is an open set around a submanifold in a smooth manifold which resembles the submanifold's normal bundle. In this paper we will discuss Henselian schemes, which are a model for tubular neighborhoods in algebraic geometry.

As discussed in Section~\ref{sec:def}, an affine Henselian scheme $\sph(A,I)$ for a Henselian pair $(A,I)$ (a ring and ideal satisfying a condition resembling Hensel's lemma) is a locally ringed space with underlying space $V(I) \subseteq \spec(A)$, and equipped with a richer sheaf of rings -- namely, on distinguished affine opens and on stalks, it is the Henselization with respect to $I$ of the structure sheaf of $\spec(A)$. These affine Henselian schemes can be glued together to form general Henselian schemes. {\it Quasi-coherent sheaves} on a Henselian scheme $(X,\O_X)$ are then defined (in Section~\ref{sec:qcoh}) as sheaves which correspond affine-locally to modules, meaning that for each $x \in X$ there is a Henselian affine open neighborhood $x \in \sph(A,I) \subset X$ such that the sheaf in question arises from an $A$-module.

Contrary to intuition (and contrary to what was originally believed), the theory of Henselian schemes in positive characteristic is better-behaved than Henselian schemes in characteristic $0$, which we will discuss in Section~\ref{sub:diff}, along with some other issues in the existing literature concerning ``Henselian smooth'' morphisms.

\subsection{Results}\label{sub:results}

We begin in Section~\ref{sec:def} by recalling the definition of Henselian schemes and (due to lack of a reference) proving that the structure presheaf on an affine Henselian scheme $\sph(A,I)$ is a sheaf for the Henselian Zariski topology; this proof uses the \Et topology on $\spec(A)$. We define quasi-coherent sheaves in Section~\ref{sec:qcoh} and prove the following:

\newcommand{\sphai}{\sph(A,I)}

\begin{restatable*}{theorem}{thmniceqcoh}\label{thm:niceqcoh} For a Henselian pair $(A,I)$ such that $A$ has characteristic $p>0$, any finitely presented (hence quasi-coherent) sheaf $\F$ on $X=\sphai$ is of the form $\widetilde{M}$ for some $A$-module $M$.  \end{restatable*}

Another factor in the suitability of Henselian schemes as ``algebraic tubular neighborhoods" \cite{cox1} is the behavior of quasicoherent sheaves in the Henselian version of the \Et topology. In order to define this {\it \hetale} topology, we must first define various properties of maps of Henselian schemes, the most basic of which is the property of being ``Henselian locally finitely presented'' or {\it \hlfp}. We give a robust definition of this property in Section~\ref{sec:lfp}, including a proof that being \hlfp is Zariski-local in Proposition~\ref{prop:affglue} (this has no $p$-torsion hypothesis). 

Although for schemes Zariski-locality for the condition of being lfp is elementary, this is not the case for \hlfp with Henselian schemes. Our proof of Proposition~\ref{prop:affglue} relies on a Henselian analogue of Grothendieck's functorial criterion for finitely presented maps \cite[Proposition 8.14.2]{ega41}, which we prove in Proposition~\ref{prop:groth}.

We then develop, without any $p$-torsion conditions, the theory of \Et and smooth morphisms of Henselian schemes in Section~\ref{sec:etsm}. There are three possible definitions we could give for a {\it Henselian smooth} (or {\it Henselian \Et}) map of affine pairs $(A,I) \to (B,J)$:
\begin{enumerate}[(i)]
\item it is the Henselization of a smooth (or \Et) algebra map (as in Definition~\ref{def:hsmet});
\item a formal lifting criterion (see Definition~\ref{def:hfsmet} for a precise formulation) supplemented by a flatness assumption;
\item it is Zariski-locally the Henselization of a smooth (or \Et) algebra map.
\end{enumerate}

We will show that (i) and (ii) are equivalent (in both the smooth and \Et cases), from which it follows that (i) can be checked Zariski-locally and so is also equivalent to (iii). We also give a concrete algebraic description in Theorem~\ref{thm:hetale} of when (ii) holds but flatness is not assumed, and use this to give an example where flatness fails (Remark~\ref{rmk:quothet}) and to deduce that such failure never occurs in the Noetherian case (Remark~\ref{rmk:noethhetale}). 

Once those notions are fully defined, we can define the {\it small Henselian \Et} (\hetale) {\it site} of a Henselian scheme and prove that quasi-coherent sheaves are sheaves for the \hetale topology:

\begin{restatable*}{theorem}{thmqcohhet}\label{thm:qcohhet} For $\F$ a quasi-coherent sheaf on a Henselian scheme $(X,\O_X)$, the presheaf $\F_{\het}$ on the small \hetale site $X_{\het}$ defined by $$(f: X' \to X) \mapsto \Gamma(X',f^*\F)$$ is a sheaf.
\end{restatable*}\vspace{.05in}

The proof of Theorem~\ref{thm:hetale} relies on a result of Arabia \cite{arabia} concerning lifting smooth algebras {\it without} localizing; these results are also useful in proving the categorical equivalence of the \hetale site of a Henselian scheme and the \Et site of the underlying ordinary scheme in Proposition~\ref{prop:globcatequiv}.

We then can prove a cohomology comparison which, like Theorem~\ref{thm:niceqcoh}, holds in positive characteristic:
\newcommand{\xhhet}{(X^h)_{\het}}
\newcommand{\xhet}{X_{\het}}
\newcommand{\fhhet}{(\F^h)_{\het}}
\newcommand{\xet}{X_{\et}}
\newcommand{\fet}{\F_{\et}}
\newcommand{\ghet}{\G_{\het}}
\newcommand{\hjet}{\H^j}
\newcommand{\hjhet}{\H^j}
\begin{restatable*}{theorem}{thmhetzarcoh}\label{thm:hetzarcoh} Let $X$ be a Henselian scheme over a Henselian pair $(A,I)$ such that $A$ has characteristic $p > 0$. Then for any quasi-coherent sheaf $\G$ on $X$ and any $i \ge 0$, the natural map $\H^i(X,\G) \to \H^i_{\het}(\xhet,\ghet)$ is an isomorphism.
\end{restatable*}

\begin{restatable*}{corollary}{corhetzarcmpsn}\label{cor:hetzarcmpsn} Let $(A,I)$ be a Henselian pair such that $A$ has characteristic $p > 0$, and let $X$ be an $A$-scheme with $I$-adic Henselization $X^h$. For a quasi-coherent sheaf $\F$ on $X$ and an integer $j$, the natural comparison map $\H^j(X,\F) \to \H^j(X^h,\F^h)$ for Zariski cohomologies is an isomorphism if and only if the \hetale comparison map $\hjet(\xet,\fet) \to \hjhet(\xhhet,\fhhet)$ is an isomorphism.
\end{restatable*}

Theorem~\ref{thm:hetzarcoh} and Corollary~\ref{cor:hetzarcmpsn} will be used in a forthcoming paper \cite{ghga} for proving a Henselian version of some formal GAGA theorems under a $p$-torsion hypothesis. 

\begin{remark}\label{rmk:ghgahet}
In \cite{ghga} we will prove a GAGA-style cohomology comparison result for coherent sheaves on the Henselization of a proper Noetherian scheme in characteristic $p>0$ by a series of reductions, including the use of an auxiliary finite morphism. For schemes, the higher direct images of a quasi-coherent sheaf by a finite morphism vanish; however, this is likely {\it not true} in the setting of Henselian schemes in general. 

By Corollary~\ref{cor:hetzarcmpsn}, we see that in positive characteristic the problem of \Et and \hetale comparison is {\it equivalent} to the comparison problem for Zariski topologies. Hence in \cite{ghga} we will instead make our reductions of GHGA statements in the setting of \Et and \hetale cohomology, because finite pushforward on abelian sheaves is {\it exact} for the \Et topology of schemes and hence for the \hetale topology of Henselian schemes. 
\end{remark}

Apart from showing quasi-coherent sheaves are sheaves for a Henselian \'etale topology in Section~\ref{sub:hetsite}, the material in Sections~\ref{sec:lfp}~and~\ref{sec:etsm} is independent of that in Section~\ref{sec:qcoh}, so it may be read before Section~\ref{sec:qcoh}. 

\renewcommand{\xhhet}{X^h}
\renewcommand{\xhet}{X}
\renewcommand{\fhhet}{\F^h}
\renewcommand{\xet}{X}
\renewcommand{\fet}{\F}
\renewcommand{\ghet}{\G}
\renewcommand{\sphai}{\sph(A)}
\renewcommand{\hjet}{\H^j_{\et}}
\renewcommand{\hjhet}{\H^j_{\het}}

\subsection{Historical difficulties}\label{sub:diff}

The theory of Henselian schemes has been studied in the past in \cite{cox1, greco71, greco81, kato17, kpr}, among others. However, there are certain subtleties to the theory that lead to issues in the existing literature on the subject concerning ``Henselian smooth'' maps and with quasi-coherent sheaves on Henselian schemes in characteristic $0$.

One such complication arises when defining what it means for a map of Henselian schemes to be ``smooth'' or ``\Et''. While in \cite{kpr} (which is cited in \cite{cox1}) flatness is deduced to be a consequence of a formal lifting criterion for {\Et}ness, this deduction has an error. In fact, flatness can fail to be a consequence. Both the error and an example of failure of flatness are discussed in Remark~\ref{rmk:kprwrong}. 

Other difficulties in the existing literature concern the case of Henselian schemes in characteristic $0$. For example, \cite[Theorem 1.12 (Theorem B)]{greco81} states that the higher cohomologies of quasi-coherent sheaves on affine Henselian schemes vanish. In Proposition~\ref{prop:cxyhglobsec}, we give a counterexample (due to de Jong) to this vanishing in characteristic $0$ for the structure sheaf on an affine Henselian scheme. In positive characteristic, de Jong has shown that higher cohomologies on affine Henselian schemes do vanish, which we discuss in Theorem~\ref{thm:thmbpos}.

This gives rise to other issues. For instance, quasi-coherent sheaves on affine Henselian schemes are {\it not} necessarily generated by their global sections in characteristic 0, as we show in Example~\ref{ex:grecowrong}. It follows that not all quasi-coherent sheaves on an affine Henselian scheme $\sph(A,I)$ arise from an $A$-module $M$, so the notion of quasi-coherence has more subtleties than was realized in earlier literature on the subject.

In contrast to many other phenomena in algebraic geometry, these pathologies are not present in positive characteristic, as we will show in Theorem~\ref{thm:niceqcoh}.

\begin{question}\label{qxn:mixthmb} Do we have ``Theorem B'' -- vanishing of higher cohomologies for quasi-coherent sheaves on affine Henselian schemes -- in {\it mixed characteristic}? \end{question}

The counterexample of Proposition~\ref{prop:cxyhglobsec} concerns a Henselian ring whose fraction field and residue field both have characteristic $0$, so it is possible that Theorem B holds in mixed characteristic. The proof of Theorem~\ref{thm:niceqcoh} ultimately relies on applying a result of Gabber \spcite{09ZI}{Theorem} concerning torsion abelian \Et sheaves, and in positive characteristic all sheaves of modules over the structure sheaf of an affine Henselian scheme are torsion. This is not the case in mixed characteristic, so the proof of Theorem~\ref{thm:niceqcoh} does not apply. Hence Question~\ref{qxn:mixthmb} remains open.

\subsection{Acknowledgments}

This paper is based on Chapters 1-5 of my Ph.D. thesis \cite{sheelathesis}. I would like to thank my Ph.D. advisor Brian Conrad, Ravi Vakil, my postdoctoral mentor Max Lieblich, and Aise Johan de Jong for useful discussions.

This material is based upon work supported by the National Science Foundation under Award No. 2102960. During the writing of my Ph.D. thesis I was supported by the NSF-GRFP fellowship (NSF Grant \#DGE-1147470) and by the Stanford Graduate Fellowship.  

\section{Defining Henselian schemes}\label{sec:def}

 A Henselian scheme is locally a Henselian affine scheme, which has as underlying space a closed subscheme of the corresponding affine scheme but with a richer sheaf of rings -- similar to a formal scheme. In this section we will give a definition of Henselian schemes which agrees with the existing literature, though a proof that the ``structure presheaf'' of a Henselian scheme is a sheaf does not seem to be given in the literature and is not ``purely algebraic'', so we prove it in Proposition~\ref{prop:isshf}.
 
A Henselian local ring $R$ is defined by the property that coprime factorizations of polynomials over its residue field lift to factorizations over $R$ itself. We can consider a global version of the same notion -- i.e., rather than a ring with a unique maximal ideal, we consider any ring and any ideal that satisfy a similar condition.

\begin{definition}\label{def:henspair} A {\bf Henselian pair} is a pair $(A,I)$ consisting of a ring $A$ and an ideal $I \subset A$ such that \begin{enumerate}[(a)]
\item $I$ is contained in the Jacobson radical $\Jac(A)$,
\item for any monic polynomial $f \in A[T]$ which factors as $\ol{f}=g_0h_0$ in $(A/I)[T]$, such that $g_0,h_0$ are monic polynomials generating the unit ideal of $(A/I)[T]$, there exist lifts of $g_0,h_0$ to monic polynomials $g,h$ in $A[T]$ with $f=gh$. 
\end{enumerate}
\end{definition}

Condition \ref{def:henspair}(a) is useful in that if $x \in A$ is a unit modulo $I$, it is a unit in $A$. The condition \ref{def:henspair}(b) is the same as the definition in the local case. 

As with Henselian local rings, there are several equivalent notions of a Henselian pair. Henselian pairs, as well as the notion of the Henselization of a pair (defined with an appropriate universal property), are discussed in detail in the Stacks Project at \spcite{09XD}{Section} and \spcite{0EM7}{Section}. 

We collect some of these definitions and other statements on Henselian pairs from the Stacks Project, as well as other useful statements which are not discussed in the Stacks Project or elsewhere, in the \hyperref[sec:appendix]{Appendix} of this paper. This includes defining Henselized polynomial rings, which are useful in Section~\ref{sec:lfp} to define Henselian finitely presented morphisms. It also includes our proof that certain ``nice'' algebraic properties can be descended from the Henselization of an algebra $R$ to some \Et $R$-algebra, which is used in Section~\ref{sec:etsm} to study Henselian smooth and Henselian \Et morphisms.

\subsection{Affine Henselian schemes}

 We will define an affine Henselian scheme for a Henselian pair $(A,I)$ as a ringed space with the underlying space $V(I) \subseteq \spec(A)$. In order to define its structure sheaf, we require the following fact:

\begin{lemma}\label{lem:defshf} Let $A$ be a ring with an ideal $I \subset A$. If $f_1,f_2 \in A$ satisfy $$D(f_1) \cap V(I) \supset D(f_2) \cap V(I),$$ then there exists a unique $A$-algebra map $A_{f_1}^h \to A_{f_2}^h$ for $A_f^h := (A_f)^h$ the $I$-adic Henselization of the principal localization $A_f$.
\end{lemma}
\begin{proof} Let $\ol{f_1},\ol{f_2}$ be the images of $f_1,f_2$ in $A/I$.

Because $D(f_2) \cap V(I) \subseteq D(f_1) \cap V(I)$, we have an $A/I$-algebra map $A_{f_1}/IA_{f_1} \to A_{f_2}/IA_{f_2}$.

Hence the image of $f_1$ in $A_{f_2}/IA_{f_2} \simeq A_{f_2}^h/IA_{f_2}^h$ is invertible. Because $A_{f_2}^h$ is Henselian, $IA_{f_2}^h\subseteq\Jac(A_{f_2}^h)$. It follows that the image of $f_1$ is invertible in $A_{f_2}^h$. 

Thus we have an $A$-algebra map $A_{f_1} \to A_{f_2}^h$. By the uniqueness condition of the universal property of Henselization, there is only one $A$-algebra map $A_{f_1}^h \to A_{f_2}^h$ compatible with the map $A_{f_1} \to A_{f_2}$. \end{proof}

We now define a presheaf of $A$-algebras on the distinguished affine opens of $V(I)$ (hence a presheaf on $V(I)$) by $\O(D(f) \cap V(I))=A_f^h$. The equality case of Lemma~\ref{lem:defshf} tells us that $\O$ is well-defined, and the general case of Lemma~\ref{lem:defshf} provides the restriction maps.

To define an affine Henselian scheme, we must show that $\O$ is a sheaf. 

\begin{proposition}\label{prop:isshf} Let $A$ be a ring containing an ideal $I$. Define a presheaf $\O$ on the base of distinguished affine opens in $V(I) \subset \spec(A)$ by $\O(D(f) \cap V(I)) =A_f^h$ where $A_f^h$ is the Henselization of the pair $(A_f,IA_f)$. Then $\O$ is in fact a sheaf.
\end{proposition}
\begin{proof} Let $Z=\spec(A/I),X=\spec(A)$, and let $\iota: Z \hookrightarrow X$ be the usual closed immersion.

Let $\O_{X_{\et}}$ be the structure sheaf on the small \Et site of $X$. Then we have a sheaf $\mc{A}=\left.\iota_{\et}^{-1}\left(\O_{X_{\et}}\right)\right|_{Z_{\zar}}$ of $A$-algebras on the Zariski site of $Z$. (Note that we are using the inverse image and not the ringed space pullback, so there are no tensor product considerations.) 

We will show that for each $f$ in $A$, that as $A$-algebras $\mc{A}(D(f) \cap V(I)) =A_f^h$, respecting the restriction maps of Lemma~\ref{lem:defshf}. This will show that $\O$ is a sheaf (in fact, $\O=\mc{A}$ on distinguished opens).

For $f \in A$ let $\ol{f}$ be the image of $f$ in $A/I$, and let $Z_{\ol{f}}=\spec(A_f/IA_f) \simeq \spec((A/I)_{\ol{f}})$. 

Let $X'=\spec(A_f^h)$ and $\phi_f: X' \to X$ be the morphism of affine schemes associated to the usual map $A \to A_f^h$. Note that $A_f/IA_f = A_f^h/IA_f^h$, so $Z_{\ol{f}}$ is a closed subscheme of $X'$. 

Then we have a commutative -- in fact Cartesian -- diagram

\begin{center}
\begin{tikzcd}
Z_{\ol{f}} \arrow[r,hook] \arrow[d,hook] & X' \arrow[d,"\phi_f"] \\
Z \arrow[r,hook,"\iota"] & X\\
\end{tikzcd}
\end{center} \vspace{-.25in}

where both horizontal arrows are closed immersions and the left vertical arrow is an open immersion. 

We see that $$\mc{A}(Z_{\ol{f}})=\Gamma\left(Z_{\ol{f}},\left.\phi_{f,\et}^{-1}\left(\O_{X_{\et}}\right)\right|_{Z_{\ol{f}}}\right)$$ by the commutativity of the diagram above and the definition of $\mc{A}$.

For an \Et morphism of schemes $\psi: Y \to X$, the map of sheaves $\psi^{-1}_{\et}\left(\O_{X_{\et}}\right) \to \O_{Y_{\et}}$ is an isomorphism. Therefore $$\Gamma\left(X',\phi_{f,\et}^{-1}\left(\O_{X_{\et}}\right)\right)= \Gamma(X',\O_{X'})=A_f^h$$ since $X'=\spec(A_f^h)$ is a cofiltered limit of \Et $X$-schemes (because $A_f$ is \Et over $A$ and $A_f^h$ is a filtered colimit of \Et $A_f$-algebras). 

Because $(A_f^h,(IA_f)^h)=(A_f^h,IA_f^h)$ is a Henselian pair, for any sheaf $\mc{G}$ on $X'_{\et}$ we have $$\Gamma(X',\mc{G})\simto\Gamma(Z_{\ol{f}},\mc{G}|_{Z_{\ol{f}}})$$ by \spcite{09ZH}{Lemma}. Applying this with $\mc{G}=\phi_{f,\et}^{-1}\left(\O_{X_{\et}}\right)$, we see that $$A_f^h \simto \Gamma\left(Z_{\ol{f}},\left.\phi_{f,\et}^{-1}\left(\O_{X_{\et}}\right)\right|_{Z_{\ol{f}}}\right)$$ so $A_f^h \simto \mc{A}(Z_{\ol{f}})$ as $A$-algebras.

Then since the underlying topological space of $Z_{\ol{f}}$ is $D(f) \cap V(I)$, we see that $$\mc{A}(D(f) \cap V(I)) =A_f^h,$$ visibly compatible with the restriction maps of $A$-algebras, as we desired to show. \end{proof}

We now define an affine Henselian scheme, and then a general Henselian scheme.

\begin{definition}\label{def:affhenssch} 
An {\bf affine Henselian scheme} $\sph(A,I)$ for a Henselian pair $(A,I)$ is a ringed space with the underlying space $V(I) \subseteq \spec(A)$, and equipped with the sheaf $\O$ given on distinguished affines by $\O(D(f) \cap V(I)) =A_f^h$ where $A_f^h$ is the Henselization of the pair $(A_f,IA_f)$ (a sheaf by Proposition~\ref{prop:isshf}). Often if there is no ambiguity we will write $\sph(A)$ for $\sph(A,I)$.

The {\it Zariski topology} on $\sph(A,I)$ is the subspace topology for the closed subscheme $V(I) \simeq \spec(A/I)$ of $\spec(A)$, which has the same underlying space as $\sph(A,I)$. 
\end{definition}

\begin{remark}\label{rmk:locringspace} The property of being Henselian is not inherited by localizations. The Henselization $A_f^h$ is generally larger than $A_f$; it follows that $\sph(A,I)$ is not a scheme. 

Though it is not a scheme, an affine Henselian scheme is a locally ringed space. The stalk at a point $\mf{p} \in V(I)$ is the local ring $(A_{\mf{p}})^h$ (where the Henselization is taken with respect to $IA_{\mf{p}}$) because Henselization commutes with direct limits.  This local ring is not a ``Henselian local ring'' in the conventional sense of that phrase, since it is not generally Henselian with respect to its maximal ideal. 

The obvious map $\sph(A,I) \to \spec(A)$ is given on local rings for $\mf{p} \in V(I)$ by $A_{\mf{p}} \to (A_{\mf{p}})^h$, which is flat. \end{remark}

\begin{remark}\label{rmk:maphens} For a Henselian pair $(A,I)$, the ring $A$ is also Henselian with respect to any ideal $J$ with $\sqrt{I}=\sqrt{J}$ \spcite{09XJ}{Lemma}, and replacing $I$ with $J$ does not affect the construction of the affine Henselian scheme $\sph(A,I)$ (see Remark~\ref{rmk:radideal}).

For $(X,\O_X)=\sph(A,I)$ an affine Henselian scheme with radical $I$, we can recover both $A$ and $I$, via $\Gamma(X,\O_X)=A$ and $I$ being the radical ideal of $\Gamma(X,\O_X)$ consisting of sections which vanish at the residue fields of every point of $X$. This is similar to how one recovers the largest ideal of definition for a Noetherian formal scheme via topologically nilpotent sections.

If $Y=\sph(B,J) \to \sph(A,I)=X$ is a morphism of locally ringed spaces with $I,J$ radical ideals of $A,B$ respectively, the associated map $\Gamma(X,\O_X) \to \Gamma(Y,\O_Y)$ gives us a map $A \to B$ which in fact is a map of pairs $(A,I) \to (B,J)$. \end{remark}

It follows from Remark~\ref{rmk:maphens} that $\sph$ is a fully faithful functor from the category of Henselian pairs with radical ideal of definition to the category of locally ringed spaces.

\subsection{Henselian schemes and the Henselization of a scheme}

\begin{definition}\label{def:henssch}
 A {\bf Henselian scheme} $(X,\O_X)$ is a ringed space which is locally isomorphic to an affine Henselian scheme, and a {\bf morphism} of Henselian schemes is a map of locally ringed spaces. Such a map can be locally described as a map of affine Henselian schemes $\sph(B,J) \to \sph(A,I)$ induced by a map of Henselian pairs $(A,I) \to (B,J)$ with $I,J$ radical ideals -- see Remark~\ref{rmk:maphens}.
\end{definition}

\begin{remark}\label{rmk:fibprod} For maps of Henselian pairs $(A,I) \to (B,J), (A,I) \to (B',J')$, the fiber product of the corresponding affine Henselian schemes $$\sph(B,J) \times_{\sph(A,I)} \sph(B',J')$$ is $\sph((B \tens_A B')^h)$, where the Henselization is taken with respect to the kernel of the natural map $B \tens_A B' \to B/J \tens_{A/I} B'/J'$. The fiber product of general Henselian schemes can be shown to exist by reducing to the affine case.

We see that for morphisms of Henselian schemes, the property of being flat on local rings is stable under base change, since if we have two maps of Henselian pairs $(A,I) \to (B,J)$ and $(A,I) \to (B',J')$, with $A \to B$ flat, then the map $B' \to B \tens_A B' \to (B \tens_A B')^h$ is also flat.
\end{remark}

Similarly to the Henselization of a pair, we can construct the Henselization of a scheme along a closed subscheme. In order to define the structure sheaf of a Henselization, we use the following proposition:

\begin{proposition}\label{prop:hensationshf} For a scheme $X$ with closed subscheme $\iota: Y \to X$, let $\O=\left.\iota_{\et}^{-1}\left(\O_{X_{\et}}\right)\right|_{Y_{\zar}}$ be a sheaf of rings on $Y$. Then for any affine open $U=\spec(A) \subset X$, we have $\O(U \cap Y) = A^h$ for $A^h$ the Henselization of $A$ along the ideal $I \subset A$ defining $U \cap Y \subset U$. 
\end{proposition}
\begin{proof} Let $U=\spec(A)$ and let $I$ be the ideal of $A$ defining $U \cap Y \subset U$. Then $U \cap Y \simeq \spec(A/I)$.

We have a commutative diagram, with both squares Cartesian:

\begin{equation}\label{diag1}\tag{$\dagger$}\begin{tikzcd}
\spec(A^h/IA^h) \arrow[d,Isom]\arrow[r,hook]&\spec(A^h)\arrow[d]\\
U \cap Y \arrow[r,hook] \arrow[d,hook] &  \arrow[d,hook]  U \\
Y \arrow[r,hook,"\iota"] & X\\
\end{tikzcd}
\end{equation} \vspace{-.25in}

such that the vertical hooked arrows are open immersions and the horizontal hooked arrows are closed immersions. The upper right vertical map is a cofiltered limit of \Et morphisms of schemes (because $A \to A^h$ is a filtered colimit of \Et ring maps). The upper left vertical map is an isomorphism since $A^h/IA^h \simeq A/I$; therefore we can consider $U \cap Y$ to be the closed subscheme of $\spec(A^h)$ defined by the ideal $IA^h$. 

Since the map $\spec(A^h) \to U \to X$ is the composition of an open immersion with a cofiltered limit of \Et morphisms of schemes, the map $\phi: \spec(A^h) \to X$ is a cofiltered limit of \Et morphisms also. Therefore $$\Gamma(\spec(A^h),\phi_{\et}^{-1}\left(\O_{X_{\et}}\right)) = \Gamma(\spec(A^h),\O_{\spec(A^h)})=A^h.$$

We now proceed as in the proof of Proposition~\ref{prop:isshf}, applying \spcite{09ZH}{Lemma} to the sheaf $\phi_{\et}^{-1}\left(\O_{X_{\et}}\right)$ and the Henselian pair $(A^h,IA^h)$. Then 

$$A^h \simto \Gamma\left(U \cap Y,\left.\phi_{\et}^{-1}\left(\O_{X_{\et}}\right)\right|_{U \cap Y}\right) = \Gamma(U \cap Y, \iota_{\et}^{-1}\left(\O_{X_{\et}}\right)) = \O(U \cap Y),$$

using the commutativity of the diagram (\ref{diag1}) for the latter equalities. \end{proof}

\begin{definition}\label{def:hensationsch}
If $X$ is a scheme with closed subscheme $Y$, the {\bf Henselization $X^h$ of $X$ along $Y$} is a locally ringed space $(X^h,\O_{X^h})$ with underlying topological space the same as that of $Y$ and $\O_{X^h}$ defined to be $\iota_{\et}^{-1}\left(\O_{X_{\et}}\right)$ where $\iota$ is the inclusion map $Y \to X$. 

By Proposition~\ref{prop:hensationshf}, for an affine open $U=\spec(A)$ of $X$ with $U \cap Y =\spec(A/I)$, the open ringed subspace of $(X^h,\O_{X^h})$ with underlying space $U \cap Y$ and structure sheaf $\O_{X^h}|_{U \cap Y}$ is isomorphic to $\sph(A^h,I^h)$. Hence $X^h$ is a Henselian scheme.

The obvious map $X^h \to X$ is flat on local rings, since this is true for $\sph(A,I) \to \spec(A)$ for any Henselian pair $(A,I)$.
 
If $X$ is a scheme over $\spec(A)$ for $(A,I)$ a Henselian pair, the {\it $I$-adic Henselization of $X$} is the Henselization of $X$ along the closed subscheme defined by $I$. For each affine open $U = \spec(B) \subset X$, the preimage of $U$ in $X^h$ is uniquely $U$-isomorphic to $\sph(B^h)$ for $B^h$ the $I$-adic Henselization of $B$. 
\end{definition}

\begin{remark}\label{rmk:getususch}
For a Henselian scheme $(X,\O_X)$, let $\mc{I}$ be the radical ideal sheaf of $\O_X$ comprising sections which vanish at the residue field of every point. Then the locally ringed space $X_0=(X,\O_X/\mc{I})$ is a scheme. We see this by considering the case $X=\sph(A,I)$, for which this process gives the scheme $\spec(A/I)$.  The general case follows by gluing the affine schemes coming from each affine open. If $X$ is the Henselization of a scheme $Y$ along a reduced closed subscheme $Z$, then the preceding construction recovers $Z$. We will revisit this in Example~\ref{ex:qcideal}. 
\end{remark}

\begin{proposition}\label{prop:hensationfinal} The Henselization $X^h \to X$ of a scheme $X$ along a closed subscheme $Y$ is final among all maps of locally ringed spaces $Z \to X$ with $Z$ a Henselian scheme such that the image of $Z$ is (topologically) contained in $Y$.
\end{proposition}
\begin{proof} This statement is Zariski-local on $X$, so we can assume that $X$ is an affine scheme $\spec(A)$ with the closed subscheme $Y \subset X$ corresponding to a radical ideal $I \subset A$.

Then what we wish to show is that $\sph(A^h,I^h) \to \spec(A)$ is final among all maps of locally ringed spaces $Z \to \spec(A)$ with $Z$ a Henselian scheme and the image topologically contained in $V(I) \subset \spec(A)$. It suffices to consider affine Henselian schemes $Z=\sph(B,J)$ for a Henselian pair $(B,J)$ such that $J$ is a radical ideal.

Any map of locally ringed spaces $Z \to X$ arises from a map of rings $\Gamma(X,\O_X)=A \to \Gamma(Z,\O_Z)=B$ \cite[Errata et Addenda, Proposition 1.8.1]{ega2}. If the image of $Z$ is topologically contained in $V(I)$, the same is true for the image of the composite map $\spec(B/J) \to Z \to X$. Hence the map of rings $A \to B$ carries the ideal $I \subset A$ into the ideal $J \subset B$, so it is in fact a map of pairs $(A,I) \to (B,J)$. 

By the universal property of the Henselization, the map of pairs $(A,I) \to (B,J)$ uniquely factors through the map of pairs $(A,I) \to (A^h,I^h)$. Therefore the map $Z \to X$ uniquely factors through the map $\sph(A^h,I^h) \to \spec(A)$, as we desired to show. \end{proof}

\section{Quasi-coherent sheaves on Henselian schemes}\label{sec:qcoh}

In this section we will prove some general facts about a notion of quasi-coherent sheaves on Henselian schemes. 

We first discuss the existing literature on the subject. The theory of quasi-coherent sheaves on affine Henselian schemes $\sph(A)$ is developed in \cite{greco81}. However, there are issues with the statements in that paper. Namely, as mentioned in Section~\ref{sec:intro}, it is claimed in \cite[Theorem 1.12 (Theorem B)]{greco81} that quasi-coherent sheaves on affine Henselian schemes have vanishing cohomology in positive degrees, but that is explicitly disproved by \cite{nothmbblog} in characteristic $0$; we reproduce this counterexample in Proposition~\ref{prop:cxyhglobsec}. In positive characteristic, this problem does not arise, as shown in \cite{thmbblog}; we state this result here in Theorem~\ref{thm:thmbpos}.

The argument in \cite{greco81} rests on the claim that quasi-coherent sheaves on affine Henselian schemes are generated by their global sections. This latter claim is also not true and can be disproved by Proposition~\ref{prop:cxyhglobsec} (\cite{nothmbblog}) as well. This will be discussed in Example~\ref{ex:grecowrong}. However, it is true in positive characteristic that ``finitely presented" quasi-coherent sheaves are generated by global sections (and hence correspond to modules), which we show in Theorem~\ref{thm:niceqcoh}.

\subsection{Definitions}

We will now give definitions for: the sheaf associated to a module, quasi-coherent sheaves, finitely presented sheaves, and coherent sheaves in the setting of Henselian schemes.

\begin{definition}\label{def:tildeshf} For an affine Henselian scheme $\sph(A,I)$ and an $A$-module $M$, the {\bf sheaf $\widetilde{M}$ associated to $M$} on $\sph(A,I)$ is defined by its values on distinguished affine opens: $$\widetilde{M}(D(f) \cap V(I)) = M \tens_A A_f^h,$$ with the evident restriction maps.

The proof that $\widetilde{M}$ is a sheaf is very similar to the proof of Proposition~\ref{prop:isshf} (replacing $\O_{X_{\et}}$ for $X=\spec(A)$ with the quasi-coherent sheaf $\F_{\et}$ on $X_{\et}$, where $\F$ is the quasi-coherent sheaf on $\spec(A)$ associated to the $A$-module $M$). Note that $M=\Gamma(\sph(A,I),\widetilde{M})$.
\end{definition}
\begin{remark}\label{rmk:kprqcoh} The sheaf $\widetilde{M}$ on $\sph(A)$ associated to an $A$-module $M$ is defined in \cite[Section 7.1.3]{kpr}, and the universal property which we will prove in Lemma~\ref{lem:tildefunctor} is proved in \cite[Satz 7.1.3.1]{kpr}. This notion is not further developed in \cite{kpr}.
\end{remark}

\begin{definition}\label{def:qcohshf} For a Henselian scheme $(X,\O_X)$, a sheaf of $\O_X$-modules $\F$ is {\bf quasi-coherent} if for each point $x \in X$, there is a Henselian affine neighborhood $U=\sph(A,I)$ of $x$ such that $\F|_U=\widetilde{M}$ for some $A$-module $M$.

This is easily seen to be equivalent to the condition that for each point $x \in X$, there is an open neighborhood $V$ of $x$ on which $\F|_V$ is isomorphic to the cokernel of a map $\bigoplus_{j \in J} \O_V \to \bigoplus_{j' \in J'} \O_V$, which is the usual definition of a quasi-coherent sheaf on a ringed space \spcite{01BK}{Lemma}.
\end{definition}
\begin{remark}\label{rmk:cxyhglobsec}
In the setting of schemes, any quasi-coherent sheaf $\F$ on an affine scheme $\spec(A)$ is generated by its global sections, i.e., it is the sheaf associated to the $A$-module $\Gamma(\spec(A),\F)$. However, it is {\it not} the case that any quasi-coherent sheaf on an affine Henselian scheme is generated by its global sections. This will be shown in Example~\ref{ex:grecowrong}, using de Jong's example of a Henselian pair $(A,I)$ such that $\H^1(\sph(A,I),\O_{\sph(A,I)})$ is nonzero \cite{nothmbblog}.

In de Jong's example, the ring $A$ is the integral closure of $(\C[x,y]^h,\mf{I}\C[x,y]^h)$ in the algebraic closure of $\C(x,y)$ (the field of rational functions), with $\mf{I}=(x^2y+xy^2-xy) \subset \C[x,y]$. In particular $A$ has characteristic $0$. We will reproduce de Jong's proof that $\H^1(\sph(A),\O_{\sph(A)}) \ne 0$ in Proposition~\ref{prop:cxyhglobsec}.

We will eventually show in Theorem~\ref{thm:niceqcoh} that when $A$ is an $\mathbf{F}_p$-algebra, any finitely presented (quasi-coherent) sheaf on $\sph(A)$ is generated by its global sections.
\end{remark}

\begin{example}\label{ex:qcideal}
For a Henselian scheme $X$, the radical ideal sheaf $\mc{I}$ defined by
$$\mc{I}(U) = \{f \in \O_X(U) \,|\, f(u) = 0 \mbox{ for all } u \in U\},$$ as in Remark~\ref{rmk:getususch},
is quasi-coherent. Indeed, on an open affine $V = \sph(B,J)$ (with $J$ a radical ideal of $B$) we have $\mc{I}(V) = J$, since the underlying space of $V$ is $\spec(B/J)$. Hence considering basic affine opens in $V$, we see that $\mc{I}|_V = \widetilde{J}$ as subsheaves of $\O_X$.

As we can obtain this (radical) ideal sheaf directly from the ringed space, we do not need to carry along an auxiliary ideal sheaf as a proxy for an adic topology satisfying the Henselian condition on affine opens. This is in contrast with morphisms of formal schemes, which are required to satisfy a continuity condition relative to the structure sheaf being a sheaf of topological rings.\end{example}

\begin{definition}\label{def:fpresshf} For a Henselian scheme $(X,\O_X)$, a sheaf of $\O_X$-modules $\F$ is {\bf finitely presented} if for each point $x \in X$, there is a Henselian affine neighborhood $U=\sph(A,I)$ of $X$ such that $\F|_U=\widetilde{M}$ for some finitely presented $A$-module $M$.
\end{definition}

\begin{definition}\label{def:cohshf} For a Henselian scheme $(X,\O_X)$, a sheaf of $\O_X$-modules $\F$ is {\bf coherent} if
\begin{enumerate}[(a)]
\item for each point $x \in X$, there is a Henselian affine neighborhood $U=\sph(A,I)$ of $X$ such that $\F(U)$ is a finitely generated $A$-module, 
\item and for all opens $V \subset X$ and all finite collections of elements $s_i \in \F(V)$, the kernel of $\bigoplus_{i=1}^n (\O_X)|_V \to \F|_V$ is of finite-type (i.e., it satisfies condition (a)).
 \end{enumerate}
\end{definition}

It is immediate from the definition that finitely presented sheaves are quasi-coherent, while a quasi-coherent sheaf is finitely presented in the sense of Definition~\ref{def:fpresshf} if and only if it is finitely presented in the sense of sheaves of modules. By general facts for ringed spaces, coherent sheaves are finitely presented \spcite{01BW}{Lemma}. The converse is true only in some situations, with suitable Noetherian hypotheses. We therefore will now define (locally) Noetherian Henselian schemes.

\begin{definition}\label{def:noethhenssch} A Henselian scheme $(X,\O_X)$ is {\bf locally Noetherian} if for each point $x \in X$, there is a Henselian affine neighborhood $U=\sph(A,I)$ of $X$ such that $A$ is Noetherian. We say that $X$ is {\bf Noetherian} if it is locally Noetherian and the underlying topological space is quasi-compact.
\end{definition}

\begin{remark}\label{rmk:noethhenssch}
If $X$ is locally Noetherian, we see by \spcite{00EO}{Lemma} that for any Henselian affine open $V=\sph(B,J) \subset X$, the ring $B$ is Noetherian. Furthermore, for a (locally) Noetherian scheme $Y$, its Henselization $Y^h$ along any closed subscheme is also (locally) Noetherian by \spcite{0AGV}{Lemma}.
\end{remark}

\begin{lemma}\label{lem:cohfpres} On a locally Noetherian Henselian scheme $(X,\O_X)$, a sheaf $\F$ of $\O_X$-modules is coherent if and only if it is finitely presented.
\end{lemma}
\begin{proof} By \spcite{01BZ}{Lemma}, it suffices to show that $\O_X$ is coherent. The proof of this fact is very similar to the scheme case \spcite{01XZ}{Lemma}, and makes use of the fact Henselization preserves the Noetherian property, as noted in Remark~\ref{rmk:noethhenssch}.\end{proof}

We will now discuss the functor from the category of $A$-modules to the category of quasi-coherent sheaves on $\sph(A)$ defined in Definition~\ref{def:tildeshf}, $M \mapsto \widetilde{M}$. In the setting of affine schemes, the corresponding functor is an equivalence of categories. However, for Henselian affine schemes, this is only true in positive characteristic. We will show in Example~\ref{ex:grecowrong} that this functor is not in general essentially surjective. 

\begin{lemma}\label{lem:tildeexact} Let $(A,I)$ be a Henselian pair and $M' \to M \to M''$ be a sequence of $A$-modules. Then the following are equivalent:
\begin{enumerate}[(i)]
\item the sequence of $A$-modules $M' \to M \to M''$ is exact, \item the sequence of quasi-coherent sheaves $\widetilde{M'} \to \widetilde{M} \to \widetilde{M''}$ on $\sph(A,I)$ is exact.
 \end{enumerate}
\end{lemma}
\begin{proof} 
We first show that (ii) $\implies$ (i). 

If $\widetilde{M'} \to \widetilde{M} \to \widetilde{M''}$ is exact, then for all maximal ideals $\mf{m}$ of $A$ we have an exact sequence $$M' \tens_A A_{\mf{m}}^h \to M \tens_A A_{\mf{m}}^h  \to M''\tens_A A_{\mf{m}}^h,$$ by talking the stalk at the closed point $\mf{m} \in \sph(A,I)$ (which we can do as $I \subseteq {\rm Jac}(A)$). Then because $A_{\mf{m}} \to A_{\mf{m}}^h$ is faithfully flat, the sequence $$M' \tens_A A_{\mf{m}} \to  M\tens_A A_{\mf{m}}  \to M''\tens_A A_{\mf{m}}$$ is also short exact. Thus the sequence of $A$-modules $M' \to M \to M''$ is exact at all maximal $\mf{m} \subset A$, hence is exact.

To show that (i) $\implies$ (ii), we need only note that for each point $\mf{p} \in \sph(A,I)$ corresponding to a prime ideal $\mf{p} \subset A$, the map $A \to A_{\mf{p}}^h$ is flat. \end{proof}

\begin{lemma}\label{lem:tildefunctor} For a Henselian pair $(A,I)$ with corresponding affine Henselian scheme $X=\sph(A,I)$ an $A$-module $M$, and any sheaf of $\O_X$-modules $\G$, we have a bijection $$\Hom_{\O_X}(\widetilde{M},\G) \to \Hom_A(M,\Gamma(X,\G)).$$
\end{lemma}
\begin{proof} This follows from the general ringed spaces case in \spcite{01BH}{Lemma}.\end{proof}

\begin{lemma}\label{lem:tildefullfaithexact} For a Henselian pair $(A,I)$ with corresponding affine Henselian scheme $X=\sph(A,I)$, the functor $M \mapsto \widetilde{M}$ from the category of $A$-modules to the category of quasi-coherent $\O_X$-modules is exact and fully faithful.
\end{lemma}
\begin{proof} This follows from Lemmas~\ref{lem:tildeexact}~and~\ref{lem:tildefunctor}.\end{proof}
\newcommand{\fwt}{\widetilde{\phantom{.}\cdot\phantom{.}}}

\begin{remark}\label{rmk:esssurjfwt}
 If it were the case that quasi-coherent sheaves on $X=\sph(A)$ are generated by their global sections, then for any quasi-coherent sheaf $\F$ on $X$, we would have an exact sequence $$\bigoplus_{j \in J} \O_X \to \bigoplus_{j' \in J'} \O_X \to \F \to 0$$ of quasi-coherent sheaves on $X$. It is easy to see that $\O_X=\widetilde{A}$, and by \spcite{01BH}{Lemma}, the functor $\fwt$ from the category of $A$-modules to the category of quasi-coherent sheaves on $X$ commutes with direct sums. Hence we would have an exact sequence $$\bigoplus_{j \in J} \widetilde{A} \to \bigoplus_{j' \in J'} \widetilde{A} \to \F \to 0.$$ By Lemma~\ref{lem:tildefullfaithexact} the functor $\fwt$ is fully faithful, so the map $\bigoplus_{j \in J} \widetilde{A} \to \bigoplus_{j' \in J'} \widetilde{A}$ arises from a map of $A$-modules $\bigoplus_{j \in J} A \to \bigoplus_{j' \in J'} A$. Let $M$ be the cokernel of this map of $A$-modules. Then since $\fwt$ is exact by Lemma~\ref{lem:tildefullfaithexact}, we would get that $\widetilde{M} \simeq \F$.

Therefore $\fwt$ is essentially surjective if all quasi-coherent sheaves on $X$ are generated by their global sections. (In fact $\fwt$ would be an equivalence of categories.) We will soon show that this is not true in general.
\end{remark}
While we have shown the functor $\fwt$ is exact and fully faithful in general, it is only essentially surjective (for finitely presented modules) in positive characteristic. As a consequence of de Jong's counterexample to Theorem B \cite{nothmbblog}, which we will describe in Proposition~\ref{prop:cxyhglobsec}, we can obtain a finitely presented counterexample to essential surjectivity of $\fwt$ in characteristic $0$. By Remark~\ref{rmk:esssurjfwt}, it follows that in characteristic $0$, quasi-coherent sheaves are not generated by their global sections. 

\begin{proposition}\label{prop:cxyhglobsec} Let $(R,J)$ be the Henselization of the pair $(\C[x,y],(xy(x+y-1)))$, and let $A$ be the integral closure of $R$ in the algebraically closed field $K=\ol{\C(x,y)}$. Let $I=JA=xy(x+y-1)A$ be the ideal of $A$ generated by $xy(x+y-1)$. Then the pair $(A,I)$ is Henselian and the cohomology group $\H^1(\sph(A,I),\O_{\sph(A,I)})$ is nonzero.
\end{proposition}
\begin{proof} This proof is due to \cite{nothmbblog} and private correspondence with Aise Johan de Jong; we reproduce it here for the convenience of the reader.

Since the map $R \to A$ is integral, it follows that $(A,I)$ is a Henselian pair \spcite{09XK}{Lemma}. Let $X=\spec(A), Z=\spec(A/I)$ and $Z_0=\spec(R/J)=\spec(\C[x,y]/(xy(x+y-1)))$, with $\iota: Z \hookrightarrow X$ the usual closed immersion.

Note that the ring $A$ is a normal domain which is integrally closed in its fraction field $K$, and $K$ is algebraically closed. This means that all local rings of $A$ are strictly Henselian \spcite{0DCS}{Lemma} and furthermore any affine scheme which is \Et over $X=\spec(A)$ is a disjoint union of opens of $X$ \spcite{09Z9}{Lemma}. 

Since $\O_{\sph(A)}=\left.\iota_{\et}^{-1}\left(\O_{X_{\et}}\right)\right|_{Z_{\zar}}$ (see the proof of Proposition~\ref{prop:isshf}) it then follows that $\O_{\sph(A)}=\O_{X}|_Z$. Thus the map from $\O_{X}$ to the constant sheaf $\underline{K}$ on $X$ gives rise to a map from $\O_{\sph(A)}$ to the constant sheaf $\underline{K}$ on $Z$. We will use the maps of sheaves $\underline{\Z} \to \O_{\sph(A)} \to \underline{K}$ on $Z$ to prove that $H^1(Z,\O_{\sph(A)})$ is nonzero.

%

We can write the extension $R \to A$ as a filtered colimit of finite normal extensions $\varinjlim A'$. For each finite normal extension $\phi: R \to A'$, we note that $R$ is regular \cite[Corollary 7.7]{greco69} and $A'$ is normal of dimension $2$, hence Cohen-Macaulay. Therefore we can apply ``miracle flatness'' \cite[Theorem 23.1]{matscrt} to see that $\phi$ is flat (on local rings) and hence is finite locally free \cite[Theorem 7.10]{matscrt}. 

By \cite[Expos\'{e} XVII, Prop. 6.2.5]{sga4} (also described in \spcite{0GKE}{Section}), if $\phi$ has degree $n$  we can obtain a ``trace'' map ${\rm Tr}: \phi_*\underline{\Z} \to \underline{\Z}$ on $(\spec(R))_\et$ such that the composition $\underline{\Z} \to \phi_*\underline{\Z} \to \underline{\Z}$ is multiplication by $n$.  

The map $Z \to Z_0$ is a cofiltered limit of the morphisms $\varphi: Z'=\spec(A'/I') \to Z_0$ for $I'=JA'$; for each $\varphi$ we can restrict the trace map described above to $Z_0$ to obtain a morphism of sheaves $\varphi_* \underline{\Z} \to \underline{\Z}$ on $(Z_0)_\et$ such that the composition $\underline{\Z} \to \varphi_*\underline{\Z} \to \underline{\Z}$ is multiplication by $n$. 

Now fix some generator $g$ of the cohomology group $\H^1(Z_0,\underline{\Z})=\Z$. We will show that $g|_Z$ is non-torsion. 

If $g$ maps to $0$ in the cohomology of $Z$, then $g$ goes to $0$ in the cohomology of some $Z'=\spec(A'/I')$ \spcite{09YP}{Lemma}. However, we have $\H^1(Z',\underline{Z})=\H^1(Z_0, \varphi_*\underline{\Z})$ for $\varphi: Z' \to Z_0$ arising from $\phi: R \to A'$, which means that if $g$ is $0$ in $\H^1(Z_0, \varphi_*\underline{\Z})$, it is annihilated by $n$ (for $n$ the degree of the finite locally free map $\phi$). 

Therefore $g|_Z$ is non-torsion, and its image in $\H^1(Z,\underline{K})$ is nonzero by the Universal Coefficient Theorem. Since we have a map of sheaves $\underline{\Z} \to \O_{\sph(A)} \to \underline{K}$ on $Z$, we see that $\H^1(Z,\O_{\sph(A)})=\H^1(\sph(A),\O_{\sph(A)}) \ne 0$. \end{proof}

\begin{example}\label{ex:grecowrong} By Proposition~\ref{prop:cxyhglobsec}, we have a Henselian pair $(A,I)$ with $A$ a $\Q$-algebra such that for $(X,\O_X)=\sph(A,I)$, the cohomology group $\H^1(X,\O_X)$ is nonzero.

For $R$ as in Proposition~\ref{prop:cxyhglobsec}, we can write $A = \varinjlim R_i$ as the filtered direct limit of its subalgebras $R_i$ which are module-finite over $R$ (and hence Noetherian). 

If $I_i=JR_i=xy(x+y-1)R_i$, then each pair $(R_i,I_i)$ is Henselian since the maps $R \to R_i$ are integral. Then $\sph(A,I) = \varprojlim \sph(R_i,I_i)$, so we can apply \cite[Lemma A.2]{ghga}\footnote{This result \cite[Lemma A.2]{ghga} depends on Lemma~\ref{lem:modpullback} which is given below, but our argument is not circular since it does not depend on Proposition~\ref{prop:cxyhglobsec} or any other results from \cite{ghga}.} to ``pull the limit into the cohomology'' and deduce $$ \varinjlim \H^1(\sph(R_i,I_i),\O_{\sph(R_i,I_i)})=\H^1(\sph(A,I),\O_{\sph(A,I)}) \ne 0.$$ Thus for some $i_0$ we have $\H^1(\sph(R_{i_0},I_{i_0}),\O_{\sph(R_{i_0},I_{i_0})}) \ne 0$. 

We may now replace $A$ with the subalgebra $R_{i_0}$, so we can assume $A$ is Noetherian. As a consequence, any coherent $\O_X$-module is finitely presented by Lemma~\ref{lem:cohfpres}.

The group $\H^1(X,\O_X)$ describes extensions of $\O_X$ by itself \spcite{0B39}{Lemma}. Therefore since $\H^1(X,\O_X) \ne 0$, we have an $\O_X$-module $\mc{N}$ that fits into a short exact sequence $\O_X \hookrightarrow \mc{N} \twoheadrightarrow \O_X$ which does not split. We see that $\mc{N}$ is coherent, hence finitely presented and quasi-coherent, by \spcite{01BY}{Lemma}.

%

Then if $\fwt$ were essentially surjective, we could get a diagram of $A$-modules $A \to N \to A$ giving rise to the short exact sequence $\O_X \hookrightarrow \mc{N} \twoheadrightarrow \O_X$ via the functor $\fwt$. Thus the diagram of $A$-modules is short exact by Lemma~\ref{lem:tildeexact}, and it is non-split by the functoriality of $\fwt$ since the associated sheaf diagram $\O_X \hookrightarrow \mc{N} \twoheadrightarrow \O_X$ is non-split by design.

However, there are clearly no non-split $A$-module extensions of $A$ by itself. This gives a contradiction, so it cannot be the case $\fwt$ is essentially surjective in general, even for coherent (or finitely presented) modules. \end{example}

We will prove the essential image of the functor $\fwt$ includes all finitely presented objects when the base ring is an $\mathbf{F}_p$-algebra in Theorem~\ref{thm:niceqcoh} after discussing the operation of pullback on quasi-coherent sheaves from a scheme to its Henselization.

\subsection{Pullback and finite pushforward}

A description of the effects of pullback and finite pushforward on sheaves on Henselian schemes will be useful for discussing $\fwt$ in positive characteristic, as previously mentioned. It will also be useful for forthcoming work on a Henselian version of formal GAGA \cite{ghga}.

We now will describe the effect of pushforward by a {\it module-finite} map on sheaves associated to modules.

\begin{lemma}\label{lem:modpush} For map of Henselian pairs $(A,I) \to (B,IB)$ which is module-finite as a map of rings, let $\phi: Y \to X$ be the corresponding map of Henselian spectra between $Y = \sph(B)$ and $X = \sph(A)$. For any $B$-module $N$, the pushforward $\phi_*(\widetilde{N})$ is the sheaf of $\O_X$-modules $\widetilde{N_A}$, for $N_A$ the underlying $A$-module of $N$ under restriction of scalars. That is, via the natural equality of $A$-modules $N_A = \Gamma(Y,\widetilde{N})=\Gamma(X,\phi_*\widetilde{N})$, the natural map $\widetilde{N_A} \to \phi_* \widetilde{N}$ via Lemma~\ref{lem:tildefunctor} is an isomorphism.
\end{lemma}
\begin{proof}  It suffices to check this on distinguished affine opens of $X$. 

For $f \in A$ with image $g \in B$, the map of rings $A_f \to A_f \tens_A B \simeq B_g$ is {\it module-finite} since $A \to B$ is. Therefore $A_f^h \tens_{A_f} B_g \simeq B_g^h$ by \spcite{0DYE}{Lemma}. 

We note that $\phi^{-1}(D(f) \cap V(I)) = D(g) \cap V(IB)$, so $\phi_*(\widetilde{N})(D(f) \cap V(I)) = N \tens_B B_g^h$. Since $B_g^h \simeq A_f^h \tens_A B$, we see that as an $A$-module, $N \tens_B B_g^h \simeq N_A \tens_A A_f^h$. Therefore $\phi_*(\widetilde{N})(D(f) \cap V(I)) = \widetilde{N_A}(D(f) \cap V(I))$ for all $f$, as we desired to show.\end{proof}

We can also describe the effect of pullback on sheaves associated to modules.

\begin{lemma}\label{lem:modpull} For a map of Henselian pairs $(A,I) \to (B,IB)$, let $\phi: Y \to X$ be the corresponding map of Henselian spectra between $Y = \sph(B)$ and $X = \sph(A)$. For any $A$-module $M$, the pullback $\phi^*(\widetilde{M})$ is the sheaf of $\O_Y$-modules $\widetilde{M \tens_A B}$. That is, via the natural map of $B$-modules $$\Gamma(X,\widetilde{M}) \tens_A B \to \Gamma(Y,\phi^*(\widetilde{M})),$$ the natural map $\widetilde{M \tens_A B} \to \phi^*(\widetilde{M})$ via Lemma~\ref{lem:tildefunctor} is an isomorphism.
\end{lemma}
\begin{proof} This follows from \spcite{01BJ}{Lemma}.\end{proof}

For ease of notation, given a sheaf $\F$ of $\O_X$-modules on a scheme $X$ with Henselization $X^h$ along some closed subscheme, we will write $\F^h$ for the pullback of $\F$ to $X^h$. (Note that the superscript $h$ might indicate pullback to a different Henselian scheme for different quasi-coherent ideal sheaves on $X$.) 

\begin{lemma}\label{lem:hensfaithexact} Let $X$ be a scheme with $i: Y \to X$ a closed subscheme, and $X^h$ the Henselization of $X$ along $Y$. Then the functor from the category of quasi-coherent sheaves on $X$ to the category of quasi-coherent sheaves on $X^h$ defined by $\F \mapsto \F^h$ is exact and faithful.
\end{lemma}
\begin{proof} In the affine case, this results from the fact that for any ring $A$ and ideal $I$, the map $\sph(A,I) \to \spec(A)$ is flat on local rings for any prime in $V(I)$ by Remark~\ref{rmk:locringspace}. The general case follows.\end{proof}

We will now give some alternative ways to compute the pullback of a sheaf on $X$ to a sheaf on $X^h$.
\begin{lemma}\label{lem:modpullback} Let $(A,I)$ be a Henselian pair and $M$ an $A$-module. Let $\F$ be the sheaf on $X=\spec(A)$ associated to $M$. Then the pullback $\F^h$ of $\F$ to $X^h$ is $\widetilde{M}$. More precisely, via the identification of $A$-modules $M=\Gamma(X,\F)$, the associated map $\widetilde{M} \to \F^h$ via Lemma~\ref{lem:tildefunctor} is an isomorphism. \end{lemma}
 \begin{proof} It suffices to show that  $\widetilde{M} \simto \F^h$ on stalks. 

For a point $\mf{p} \in X$, the local ring of $\O_X$ at $\mf{p}$ is $A_{\mf{p}}$. If $\mf{p} \in V(I)$, it can also be thought of as a point of $X^h$, with local ring $A_\mf{p}^h$ (the $I$-adic Henselization of $A_{\mf{p}}$).

The stalk $(\F^h)_{\mf{p}}$ is isomorphic to $M_\mf{p} \tens_{A_{\mf{p}}} A_{\mf{p}}^h = M \tens_A A_{\mf{p}} \tens_{A_{\mf{p}}} A_{\mf{p}}^h \simeq M \tens_A A_{\mf{p}}^h$. 

By the definition of $\widetilde{M}$, we see that the stalk $(\widetilde{M})_{\mf{p}}$ is isomorphic to $M \tens_A A_{\mf{p}}^h$, since Henselization commutes with direct limits. Then as we desired to show, $\widetilde{M} \simto \F^h$ at each stalk.\end{proof}

\begin{lemma}\label{lem:fancypullback} Let $X$ be a scheme with $i: Y \to X$ a closed subscheme, and $X^h$ the Henselization of $X$ along $Y$. Then for a quasi-coherent sheaf $\F$ on $X$, the pullback $\F^h$ of $\F$ to $X^h$ is isomorphic to $(i_{\et}^{-1}(\F_{\et}))|_{Y_{\zar}}$ for $\F_{\et}$ the sheaf on the small \Et site of $X$ corresponding to $\F$ and $i_{\et}$ the morphism of sites $Y_{\et} \to X_{\et}$ corresponding to $i$. 
\end{lemma}
\begin{proof}
We proceed similarly to the proof of Proposition~\ref{prop:hensationshf}. 

Motivated by Proposition~\ref{prop:hensationshf}, let $\G=(i_{\et}^{-1}(\F_{\et}))|_{Y_{\zar}}$, which is a sheaf of $i_{\et}^{-1}(\O_{X_{\et}})|_{Y_{\zar}}$-modules. In fact the sheaf $(i_{\et}^{-1}(\O_{X_{\et}}))|_{Y_{\zar}}$ is the structure sheaf of $X^h$, which has the same underlying topological space as $Y$, by Definition~\ref{def:hensationsch}. Hence we can consider $\G$ as a sheaf of $\O_{X^h}$-modules. We will define a map of sheaves of $\O_{X^h}$-modules $\F^h \to \G$ and prove that it is an isomorphism.

For $p$ the canonical map of ringed spaces $X^h \to X$, we see that $p^{-1}\F$ is the sheafification of the presheaf $$U \mapsto \colim_{\substack{V \subset X \;\text{open}\\U \subset V}} \F(V)$$ on $Y$. Similarly, $i_{\et}^{-1}(\F_{\et})$ is the sheafification of the presheaf $$T \mapsto \colim_{\substack{X' \to X\;\\T \to X' \times_X Y}} \F_{\et}(X')$$ on the small \Et site $Y_{\et}$.

For each open subset $U$ of $Y$, there is a natural map $$\colim_{\substack{V \subset X \;\text{open}\\U \subset V}} \F(V) \to \colim_{\substack{X' \to X\;\text{\Et}\\U \to X' \times_X Y}} \F_{\et}(X'),$$ arising from the fact that any open $V \subset X$ with $U \subset V$ can be considered as an \Et $X$-scheme via the open immersion $V \hookrightarrow X$, with an obvious morphism $U \to V \times_X Y = V \cap Y$ (which exists since $U \subset Y$).

This map is compatible with overlaps of open subsets of $Y$ (or open subsets of $X^h$, which has the same underlying topological space). Therefore it gives rise to a map of sheaves $p^{-1}\F \to i_{\et}^{-1}(\F_{\et})|_{Y_{\zar}}=\G$ which is linear over $p^{-1}\O_X \to i_{\et}^{-1}(\O_{X_{\et}})|_{Y_{\zar}}=\O_{X^h}$, so we have a map of sheaves of $\O_{X^h}$-modules $\F^h \to \G$. We will prove that this map is an isomorphism by checking on stalks.

Let $\mc{I}$ be the ideal sheaf defining $Y$ in $X$. For a point $y \in Y \subset X$, let $A_y=\O_{X,y}$ be the local ring of $y$ in $X$, and $\mc{I}_y$ be the stalk of $\mc{I}$ at $y$, which is an ideal of $A_y$. The quotient ring $A_y/\mc{I}_y$ is the local ring $\O_{Y,y}$ of $y$ in $Y$.

As $X^h$ has the same underlying topological space as $Y$, we can consider $y$ as a point of $X^h$, with local ring $A_y^h$ the $\mc{I}_y$-adic Henselization of $A_y$.

We have a commutative diagram, with both squares Cartesian:

\begin{equation}\label{diag2}\tag{$\dagger\dagger$}
\begin{tikzcd}
\spec(A_y^h/\mc{I}_yA_y^h) \arrow[d,Isom]\arrow[r,hook]&\spec(A_y^h)\arrow[d,"\psi_y"]\\
\spec(A_y/\mc{I}_y)\arrow[r,hook] \arrow[d] &  \arrow[d,"\iota_y"]  \spec(A_y) \\
Y \arrow[r,hook,"i"] & X\\
\end{tikzcd} 
\end{equation} \vspace{-.25in}

such that the horizontal hooked arrows are closed immersions. 

The upper left vertical map is an isomorphism since $A_y^h/\mc{I}_yA_y^h \simeq A_y/\mc{I}_y$.

Since $\F$ is quasi-coherent and the upper right vertical map $\psi_y$ is a cofiltered limit of \Et morphisms of schemes (because $A_y \to A_y^h$ is a filtered colimit of \Et ring maps). Then if we let $\vp_y=\iota_y\psi_y,$
we obtain \begin{equation}\label{eq1}\Gamma(\spec(A_y^h),\vp_{y,\et}^{-1}\left(\F_{\et}\right)) = \Gamma(\spec(A_y),\iota_{y,\et}^{-1}(\F_{\et})) \tens_{A_y} A_y^h.\end{equation} Since $A_y=\O_{X,y}$ is a local ring of $X$ at the point $y$, we compute that \begin{equation}\label{eq0} \Gamma(\spec(A_y),\iota_{y,\et}^{-1}(\F_{\et}))=\F_y, \end{equation} 
for $\F_y$ the stalk of $\F$ at $y$. Then from equation (\ref{eq0}) we see that \begin{equation}\label{eq2} \Gamma(\spec(A_y),\iota_{y,\et}^{-1}(\F_{\et})) \tens_{A_y} A_y^h=\F_y \tens_{A_y} A_y^h =(\F^h)_y,\end{equation} the stalk of $\F^h$ at $y \in X^h$. 

Let $\G'= \vp_{y,\et}^{-1}\left(\F_{\et}\right)$. Combining equations (\ref{eq1}) and (\ref{eq2}), we see that \begin{equation}\label{eq3}\Gamma(\spec(A_y^h),\G') \simeq (\F^h)_y.\end{equation}
By commutativity of the diagram (\ref{diag2}), we see that \begin{equation}\label{eq4}\Gamma(\spec(A_y/\mc{I}_y), i_{\et}^{-1}\left(\F_{\et}\right)) = \Gamma\left(\spec(A_y^h/\mc{I}_yA_y^h),\left.\G'\right|_{\left(\spec(A_y^h/\mc{I}_yA_y^h)\right)_{\et}}\right).\end{equation} For any sheaf $\mc{H}$ on $Y$, we note that $\mc{H}(\spec(A_y/\mc{I}_y)) = \mc{H}_y$ is the stalk of $\mc{H}$ at $y$. Then since $\G(\spec(A_y/\mc{I}_y)) = \Gamma(\spec(A_y/\mc{I}_y), i_{\et}^{-1}\left(\F_{\et}\right))$ by the definition $\G:=(i_{\et}^{-1}(\F_{\et}))|_{Y_{\zar}}$, we have \begin{equation}\label{eq5} \G_y=\Gamma\left(\spec(A_y^h/\mc{I}_yA_y^h),\left.\G'\right|_{\left(\spec(A_y^h/\mc{I}_yA_y^h)\right)_{\et}}\right) \end{equation} by the equation (\ref{eq4}).
Because $(A_y^h,\mc{I}_yA_y^h)$ is a Henselian pair, we see that $$\Gamma(\spec(A_y^h),\G')\simto\Gamma\left(\spec(A_y^h/\mc{I}_yA_y^h),\left.\G'\right|_{\left(\spec(A_y^h/\mc{I}_yA_y^h)\right)_{\et}}\right)$$ by \spcite{09ZH}{Lemma}. Therefore $$ (\F^h)_y  \simeq \Gamma(\spec(A_y^h),\G') \simto \Gamma\left(\spec(A_y^h/\mc{I}_yA_y^h),\left.\G'\right|_{\left(\spec(A_y^h/\mc{I}_yA_y^h)\right)_{\et}}\right) \simto \G_y$$ by combining the equations (\ref{eq3}) and (\ref{eq5}), so the stalks of $\F^h$ and $\G$ at $y$ are isomorphic (via the natural map) as we desired to show.\end{proof}

Now that we have described the pullback $\F^h$ in more detail, we can prove that in positive characteristic the essential image of $\fwt$ contains all finitely presented objects. We will need an equivalent of ``Theorem B'' for Henselian schemes in positive characteristic.

\begin{theorem}\label{thm:thmbpos} Let $(A,I)$ be a Henselian pair such that $A$ has characteristic $p > 0$. Then if $Z=\sph(A)$, for any quasi-coherent sheaf $\F$ on $X=\spec(A)$, the cohomologies $\H^j(Z,\F^h)$ are $0$ for $j > 0$. 
\end{theorem}

Theorem~\ref{thm:thmbpos} was proved by de Jong in the Stacks Project Blog post at \cite{thmbblog}; we reproduce that proof here for the convenience of the reader.

Before proving Theorem~\ref{thm:thmbpos}, we need the following lemma.

\begin{lemma}\label{lem:thmblem} For $S$ an affine scheme and $\G$ an abelian sheaf on the small \Et site $S_{\et}$ of $S$, if for every affine object $T$ in $S_{\et}$ and all $i > 0$ the \Et cohomology groups $\H^i_{\et}(T,\G)$ vanish, then the Zariski cohomology groups $\H^i(S,\G)$ also vanish for all $i > 0$.
\end{lemma}\begin{proof}[Proof of Lemma~\ref{lem:thmblem}]
For any affine open $S' \subset S$ with a finite affine open cover $S_1' \cup \dots \cup S_n'= S'$, we note that all intersections of the $S_j'$ are also affine. 

Therefore by hypothesis, the \v{C}ech-to-cohomology spectral sequence for $\G_{\et}$ on the cover $\bigcup_{i=1}^n S_i'= S'$ in the \Et topology degenerates, so the \v{C}ech complex of this cover is exact in positive degrees. 

However, this \v{C}ech complex can also be used to compute the Zariski cohomology of $\G$ on $S'$; therefore by \spcite{01EW}{Lemma}, we have $\H^i(S,\G)=0$ for all $i > 0$ as we desired to show.\end{proof}

\begin{proof} [Proof of Theorem~\ref{thm:thmbpos}]

Let $Y=\spec(A/I)$; then $Y$ and $Z=X^h$ have the same underlying topological space. For $i: Y \hookrightarrow X$ the canonical closed immersion (with $i_{\et}: Y_{\et} \to X_{\et}$ the corresponding morphism of small \Et sites), the sheaf $(i_{\et}^{-1}(\F_{\et}))|_{Y_{\zar}}$ on $Y$ is in fact a sheaf of $\O_Z$-modules, and is equal to the pullback $\F^h$ of $\F$ to $Z$ by Lemma~\ref{lem:fancypullback}.

Any affine object $V \in Y_{\et}$ takes the form $V=\spec(B/IB)=\spec(B^h/IB^h)$, where $B$ is an \Et $A$-algebra. We will show that $\H^j_{\et}(V,i_{\et}^{-1}(\F_{\et}))=0$ for $j>0$. Since $(B^h,IB^h)$ is a Henselian pair and $\F,\F^h$ are both necessarily $p$-torsion, by \spcite{09ZI}{Theorem} we have $$\H^j_{\et}(\spec(B^h),\F) = \H^j_{\et}(V,i_{\et}^{-1}(\F))$$ for all $j$. Then since $\H^j_{\et}(\spec(B^h),\F)=0$ for $j > 0$, we have $\H^j_{\et}(V,i_{\et}^{-1}(\F_{\et}))=0$ as we desired to show, for any affine object $V \in Y_{\et}$.

Recall that $Y$ and $Z$ have the same underlying space. Then since $\H^j_{\et}(V,i_{\et}^{-1}(\F_{\et}))=0$ for any affine $V \in Y_{\et}$ and any $j > 0$, we 
apply Lemma~\ref{lem:thmblem} to $i_{\et}(\F_{\et})$ to see that the Zariski cohomology $H^j(Z,\F^h)$ vanish for all $j > 0$ as well. \end{proof}

We can now apply Theorem~\ref{thm:thmbpos} to prove that in positive characteristic the essential image of $\fwt$ contains all finitely presented objects, as mentioned in Section~\ref{sec:intro}. By Lemma~\ref{lem:tildefullfaithexact}, in fact $\fwt$ is an equivalence of categories of finitely presented objects, as in the case of ordinary schemes. This is only true in positive characteristic for Henselian schemes, however.

\thmniceqcoh

\begin{proof} Since $\F$ is finitely presented, it is quasi-coherent. For each maximal ideal $\mf{m}$ of $A$, we have a short exact sequence $$0 \to \mf{m}\F \to \F \to \F/\mf{m}\F \to 0$$ of quasi-coherent sheaves on $X$. We can consider $\F/\mf{m}\F$ as a quasi-coherent sheaf on $\sph(A/\mf{m})$ of finite type. 

Since $A/\mf{m}$ is a field, we see that $\sph(A/\mf{m})=\spec(A/\mf{m})$ as locally ringed spaces. Hence $\F/\mf{m}\F$ arises from a finite-dimensional $A/\mf{m}$-vector space $V$. 

By Theorem~\ref{thm:thmbpos}, the cohomology group $\H^1(X,\mf{m}\F)$ vanishes. Hence the map of global sections $\Gamma(X,\F) \to \Gamma(X,\F/\mf{m}\F)$ is surjective. We can choose a finite set of global sections $f_1,\dots,f_r$ of $\F$ which reduce to a spanning set of $\F/\mf{m}\F$ over $A/\mf{m}$. By Nakayama's Lemma, these global sections generate the stalk $\F_{\mf{m}}$ over the local ring $\O_{X,\mf{m}}=(A_{\mf{m}})^h$. Since $\F$ is of finite type, the restrictions $f_1|_{U},\dots,f_r|_{U}$ generate $\F|_U$ for some open neighborhood $U$ of $\mf{m}$. 

Therefore for each closed point $\mf{m}$ of $X$, we can find an open neighborhood $U$ of $\mf{m}$ such that $\F|_U$ is generated by finitely many global sections. 

The union of all such $U$ contains all the closed points of $X$, and hence covers $X$. Then since $X$ is quasi-compact, we can choose finitely many open subsets $U$ which cover $X$ and such that $\F|_U$ is generated by finitely many global sections. Hence for some large $N$, we have a surjection $\psi:\O_{X}^N \twoheadrightarrow \F$. 

Because $\F$ is finitely presented, the kernel of $\psi$ is some quasi-coherent sheaf $\G$ of finite type. Note that we only used the fact that $\F$ was finite type to construct the surjection $\psi$; hence we can apply the above to get a surjection $\psi': \O_{X}^{N'} \twoheadrightarrow \G$ for some large $N'$. Then we have an exact sequence $$\O_{X}^{N'} \to \O_{X}^N \twoheadrightarrow \F.$$ The map $\O_{X}^{N'} \to \O_{X}^N$ arises from some map of free $A$-modules $A^{N'} \to A^N$ by Lemma~\ref{lem:tildefullfaithexact}. If $M$ is the cokernel of this map $A^{N'} \to A^N$, we see (again by Lemma~\ref{lem:tildefullfaithexact}) that $\F = \widetilde{M}$. \end{proof}

We now use formal GAGA to show that the global sections of $\F^h$ over $X^h$ are isomorphic to the global sections of $\F$ over $X$ for a scheme $X$ which is proper over a Noetherian complete affine base. This will allow us to prove that the functor $\F \mapsto \F^h$ is fully faithful in the proper case over a complete Noetherian affine base.
 
 \begin{lemma}\label{lem:h0gen} Let $(A,I)$ be a Noetherian Henselian pair with $I$-adically complete $A$, and let $X$ be a proper $A$-scheme with $I$-adic Henselization $X^h$. Then for any coherent sheaf $\F$ on $X$, the natural map $\Gamma(X,\F) \to \Gamma(X^h,\F^h)$ is an isomorphism. \end{lemma}
\begin{proof} 

We define the formal scheme $X^\wedge$ over $\spf(A)$ as the $I$-adic completion of $X$. We consider the maps $\Gamma(X,\F) \to \Gamma(X^h,\F^h) \to \Gamma(X^\wedge,\F^\wedge)$. The composite map is an isomorphism by formal GAGA \cite[Theorem 5.1.4]{ega31}. Therefore the map $\Gamma(X,\F) \to \Gamma(X^h,\F^h)$ is injective.

We can see that the map $\Gamma(X^h,\F^h) \to \Gamma(X^\wedge,\F^\wedge)$ is also injective by working over open affines $\spec(B) \subset X$. Because such $B$ are Noetherian, the same is true of $B^h$. Thus the map  $B^h \to B^\wedge$ is faithfully flat since $IB^h \subseteq {\rm Jac}(B^h)$ and $B^h/IB^h \simto B^\wedge/IB^\wedge$. This implies that the map $$M \tens_B B^h \to M^\wedge := M \tens_B B^\wedge$$ is injective. 

Then since $$\Gamma(X,\F)  \hookrightarrow \Gamma(X^h,\F^h) \hookrightarrow \Gamma(X^\wedge,\F^\wedge)$$ is a composition of injective maps which is an isomorphism, each of the injections is an isomorphism. Hence $\Gamma(X,\F) \simto \Gamma(X^h,\F^h)$ as we desired to show.\end{proof}

\begin{proposition}\label{prop:hensfullfaithexact}  Let $(A,I)$ be a Noetherian Henselian pair with $I$-adically complete $A$, and let $X$ be a proper $A$-scheme with $I$-adic Henselization $X^h$. Then the functor between the category of coherent sheaves on $X$ and the category of coherent sheaves on $X^h$ defined by $\F \mapsto \F^h$ is exact and fully faithful.
\end{proposition}
\begin{proof} We already know that this functor is exact and faithful by Lemma~\ref{lem:hensfaithexact}. 

Now consider two coherent sheaves $\F,\G$ on $X$. The coherent Hom-sheaf $\mc{H}om_X(\F,\G)$ has Henselization $\mc{H}om_{X^h}(\F^h,\G^h)$ since $X^h \to X$ is flat on local rings. By Lemma~\ref{lem:h0gen}, we deduce that the natural map $$\Gamma(X,\mc{H}om_X(\F,\G)) \to \Gamma(X^h,\mc{H}om_{X^h}(\F^h,\G^h))$$ is an isomorphism. By the definition of the Hom-sheaf, this is the natural map $$\Hom_X(\F,\G) \to \Hom_{X^h}(\F^h,\G^h),$$ so that map is bijective.

Therefore the functor $\F \mapsto \F^h$ is fully faithful as we desired to show.\end{proof}

\begin{remark}\label{rmk:hensfullfaithexactext}
In view of Proposition~\ref{prop:hensfullfaithexact}, it is natural to wonder when the functor $\F \mapsto \F^h$ is an equivalence of categories. In the affine case, by Lemma~\ref{lem:modpullback} the essential surjectivity of the functor $\F \mapsto \F^h$ is equivalent to the question of essential surjectivity of $\fwt$. This does not hold even for finitely presented modules, as seen in Example~\ref{ex:grecowrong}, which has a Henselian pair $(A,I)$ where $A$ is a $\Q$-algebra. 

These issues do not arise for affine Henselian schemes $\sph(A,I)$ where $A$ is an $\mathbf{F}_p$-algebra, as discussed in Theorem~\ref{thm:niceqcoh}. It is not clear if an affine characteristic $0$ counterexample similar to Proposition~\ref{prop:cxyhglobsec} and Example~\ref{ex:grecowrong} exists where $A$ is not a $\Q$-algebra or a $\mathbf{F}_p$-algebra. 

However, de Jong showed in \cite{cohpropblog} that even for a complete DVR $A$ with fraction field $K$ of characteristic $0$, the group $\H^1((\P_A^1)^h,\O_{(\P_A^1)^h})$ is nonzero. This nonvanishing is used in \cite[Example 4.4.1]{ghga} to give a similar algebraizability counterexample to Example~\ref{ex:grecowrong}. Thus even when the base ring is $A=\Z_p$, which is not a $\Q$-algebra, Henselian schemes in characteristic 0 exhibit pathological behavior.

In \cite{ghga} we will show that algebraizability is inherited by coherent subsheaves (without any $p$-torsion hypotheses). Applying this to ideal sheaves, we will deduce a Henselian equivalent of Chow's theorem on algebraization and the algebraizability of maps between Henselizations of proper schemes; we will also be able to extend Proposition~\ref{prop:hensfullfaithexact} to the case where the base ring $A$ is Noetherian but not necessarily complete.
\end{remark}

\section{Morphisms locally of finite presentation}\label{sec:lfp}

In order to develop a theory of  \Et and smooth maps between Henselian schemes, it is first necessary to have a robust notion of ``finite presentation''.  Because a finitely presented algebra over a Henselian ring is not generally Henselian itself, we cannot use the usual definition from the scheme case.

Our definition is similar to that in \cite{cox1}; however, it is necessary to verify certain useful properties of maps of ``Henselian finite presentation'' that are not checked in \cite{cox1}. For example, we would like to be able to check whether a map is ``Henselian finitely presented'' Zariski-locally, for which it will be useful to have a functorial criterion similar to that of \cite[Corollary 8.14.2.2]{ega43} for whether a map is ``Henselian finitely presented''. We will prove such a criterion in Proposition~\ref{prop:groth}; this is not discussed in \cite{cox1}.

\subsection{Basic concepts and properties}\label{sub:basichfp}

In the \hyperref[sec:appendix]{Appendix} we define for a Henselian pair $(A,I)$ the {\bf Henselized polynomial ring} $A\{X_1,\dots,X_n\}$ (Definition~\ref{def:henspoly}) and show in Remark~\ref{rmk:hfpex} and Lemma~\ref{lem:polyonpoly} that this notion behaves well with regards to taking the quotient by a finitely generated ideal or adjoining additional variables.

\begin{definition}\label{def:h(l)fp}
A map of Henselian pairs $\phi: (A,I) \to (B,J)$ is {\bf Henselian finitely presented} or {\bf \hfp} if $\sqrt{J}=\sqrt{IB}$ and $\phi$ gives an isomorphism $B \simeq A\{X_1,\dots,X_n\}/(g_1,\dots,g_m)$ for $g_i \in A\{X_1,\dots,X_n\}$. 

A map of Henselian schemes $f: (X,\O_X) \to (Y,\O_Y)$ is {\bf Henselian locally finitely presented} or {\bf \hlfp} if $f$ can be Zariski-locally described as $\widetilde{\phi}: \sph(B,J) \to \sph(A,I)$ where the corresponding map of pairs $\phi: (A,I) \to (B,J)$ is \hfp.
\end{definition}

\begin{remark}\label{rmk:rad}
Since $\sqrt{J} = \sqrt{IB}$ implies that $(B,IB)$ is also a Henselian pair, for any map of Henselian pairs $(A,I) \to (B,J)$ which is \hfp it is harmless to arrange that $J=IB$. 
\end{remark}

\begin{example}\label{ex:disthfp}
It will often be useful to note that for a Henselian pair $(A,I)$ and an element $f \in A$, we have $A_f^h \simeq A\{T\}/(1-Tf)$ by Remark~\ref{rmk:hfpex}; this is a basic, useful example of an \hfp $A$-algebra. Therefore distinguished affine opens are \hfp. \end{example}

We first show that an \hfp algebra over a base pair $(A,I)$ is  the $I$-adic Henselization of a finitely presented $A$-algebra.

\begin{lemma}\label{lem:hfpishoffp} Let $\phi: (A,I) \to (B,IB)$ a map of Henselian pairs. Then the following are equivalent: 
\begin{enumerate}[(i)]
\item the map $\phi$ is \hfp,
\item there exists a finitely presented $A$-algebra $R$ such that $(B,IB) \simeq (R^h,IR^h)$ as $A$-algebras.\end{enumerate}
\end{lemma}
\begin{proof} 

We first show that (i) $\implies$ (ii). From the definition of ``\hfp'', we have a presentation $B \simeq A\{X_1,\dots,X_n\}/(g_1,\dots,g_m)$. 

Because $A\{X_1,\dots,X_n\}$ is a filtered colimit of \Et $A[X_1,\dots,X_n]$-algebras (by the definition of Henselization), we can find an \Et $A[X_1,\dots,X_n]$-algebra $A'$ with $A'/IA' \simeq A[X_1,\dots,X_n]/IA[X_1,\dots,X_n] $ (hence $(A')^h = A\{X_1,\dots,X_n\}$ by Proposition~\ref{prop:hensetexp}) so that for all $j$:
\begin{enumerate}[\hspace{2ex}(a)]
\item the element $g_j \in A\{X_1,\dots,X_n\}$ arises from some $g_j' \in A'$,
\item if $\Phi: A' \to A\{X_1,\dots,X_n\} \to B$ is the composition of the natural map with the quotient map, then $\Phi(g_j')=0$. 
\end{enumerate}

Now let $R=A'/(g_1',\dots,g_m')$. Because the $g_j$ arise from the $g_j'$, we see that \begin{equation}\label{star}(A')^h \tens_{A'} R \simeq A\{X_1,\dots,X_n\} \tens_{A'} R \simeq B.\tag{$\star$}\end{equation}

Thus by equation (\ref{star}), we see that $R^h \simeq B$ over $A$ (Lemma~\ref{lem:hensmodfin}).

Since $R$ is finitely presented over an \Et $A[X_1,\dots,X_n]$-algebra $A'$, it is evidently finitely presented over $A$. Hence $B$ is the Henselization of a finitely presented $A$-algebra as we desired to show. 

The fact that (ii) $\implies$ (i) is an easy consequence of Remark~\ref{rmk:hfpex}. \end{proof}

\begin{remark}\label{rmk:hlfpbasechange}
For a map of Henselian pairs $(A,I) \to (B,J)$ and an $A$-algebra $R$ (with $I$-adic Henselization $R^h$), we note that the $J$-adic Henselizations of the $B$-algebras $B \tens_A R^h$ and $B \tens_A R$ are canonically isomorphic by considering the uniqueness conditions of the universal property of Henselization and the universal property of tensor products. 

Given any \hfp map $A \to C$, we can find a finitely presented $A$-algebra $S$ such that $C \simeq S^h$ the $I$-adic Henselization of $S$  (Lemma~\ref{lem:hfpishoffp}). Then by the preceding paragraph, the $J$-adic Henselization $(B \tens_A C)^h$ is the same as the $J$-adic Henselization of the finitely presented $B$-algebra $B \tens_A S$, so the map $B \to (B \tens_A C)^h$ is \hfp. Hence the property of being \hfp (or \hlfp) is stable under base change.
\end{remark}

We now will show that the condition of being \hfp is compatible with composition of maps.

\begin{proposition}\label{prop:hfpcomp} The composition of \hfp maps is \hfp, and the composition of \hlfp maps is \hlfp.
\end{proposition}
\begin{proof} By the definitions and Example~\ref{ex:disthfp}, it suffices to show that the composition of \hfp maps is \hfp.

Let $(A,I) \to (B,IB) \to (B',IB')$ be two \hfp maps of Henselian pairs. 

For $B \simeq A\{X_1,\dots,X_n\}/(g_1,\dots,g_m), B' \simeq B\{Y_1,\dots,Y_\ell\}/(h_1,\dots,h_r)$, we have $$B' \simeq A\{X_1,\dots,X_n,Y_1,\dots,Y_\ell\}/(g_1,\dots,g_m,h_1,\dots,h_r)$$ since $$A\{X_1,\dots,X_n\}\{Y_1,\dots,Y_\ell\} \simeq A\{X_1,\dots,X_n,Y_1,\dots,Y_\ell\}$$ by Lemma~\ref{lem:polyonpoly}. Therefore $B'$ is \hfp over $A$.\end{proof}

\begin{proposition}\label{prop:hlfphlfp} If $(A,I) \to (B,IB), (A,I) \to (C,IC)$ are two \hfp maps, then any map of Henselian pairs $\Phi: (B,IB) \to (C,IC)$ over $(A,I)$ is \hfp.
\end{proposition}
\begin{proof} We can choose $R,S$ finitely presented $A$-algebras with $R^h \simeq B, S^h \simeq C$ by Lemma~\ref{lem:hfpishoffp}. Then $\Phi$ gives a map $R \to C$. Writing $C$ as the colimit of \Et $S$-algebras $S'$ such that $S/IS \simeq S'/IS'$, we can choose $S_1$ \Et over $S$ with $S/IS \simeq S_1/IS_1$ and $\ol{\Phi}: R \to S_1$ over $A$ so that the map  $R \to S_1 \to (S_1)^h \simeq C$ is $\Phi$ by \spcite{00QO}{Lemma}. (The last isomorphism is due to Proposition~\ref{prop:hensetexp}.)

Since $S_1$ is \Et over $S$ which is finitely presented over $A$, then $S_1$ is also finitely presented over $A$. Therefore $\ol{\Phi}$ is a map of finitely presented $A$-algebras, hence it is also a finitely presented map.

Then we can write $S_1 \simeq R[T_1,\dots,T_n]/(g_1,\dots,g_m)$. We have a map $S_1 \to B \tens_R S_1 \to (B \tens_R S_1)^h$, which must factor through $C \simeq S_1^h$ by the universal property of Henselization. Therefore we have an $S_1$-algebra map $\eta: C \to (B \tens_R S_1)^h$.

The universal property of the tensor product gives us a map $B \tens_R S_1 \to C$ via the $R$-algebra maps $B \to C, S_1 \to C$; this map factors through $(B \tens_R S_1)^h$ because of the universal property of Henselization. Hence we have a map $\mu: (B \tens_R S_1)^h \to C$.

By the uniqueness condition of the universal property of Henselization, since $\mu\eta: C \to C$ is an $S_1$-algebra map from $C \simeq (S_1)^h$ to itself, it is the identity. Similarly, $\eta\mu$ is the unique $(B \tens_R S_1)$-algebra map from $(B \tens_R S_1)^h$ to itself, so it is the identity.

Therefore $C \simeq (B \tens_R S_1)^h$, which is the Henselization of the finitely presented $B$-algebra $B \tens_R S_1$. Hence $\Phi$ is \hfp as we desired to show.\end{proof}

We will now prove a useful functorial criterion for showing that a map of Henselian pairs is \hfp. This is analogous to \cite[Proposition 8.14.2]{ega41}, though our criterion will be used to show in Proposition~\ref{prop:affglue} that the condition of being \hfp can be checked Zariski-locally, whereas in the scheme case the analogous functorial criterion is not needed for the analogous purpose.

\begin{proposition}\label{prop:groth} Consider a map of rings $f: A \to B$ where $(A,I)$ and $(B,IB)$ are Henselian pairs. The following are equivalent:

\begin{enumerate}[(i)]
\item $f$ is \hfp,
\item for any inductive filtered system $(R_\alpha)$ of Henselian pairs over $(A,I)$ (each $R_\alpha$ is an $A$-algebra with $(R_\alpha,IR_\alpha)$ a Henselian pair), the canonical map $$\colim \Hom_A(B,R_\alpha) \to \Hom_A(B,\colim R_\alpha)$$ is bijective.

\end{enumerate}

\end{proposition}
\begin{proof}

We will first show that condition (i) implies condition (ii).

Assume that $f$ is \hfp; by Lemma~\ref{lem:hfpishoffp} we can write $B$ as the Henselization of $(B_0,IB_0)$ for a finitely presented $A$-algebra $B_0=A[X_1,\dots,X_n]/(g_1,\dots,g_m)$, giving a ``Henselian finite presentation'' of $B$. Since $R_\alpha$ is Henselian for each $\alpha$, then $\colim R_\alpha$ is Henselian with respect to the ideal $I\left(\colim R_\alpha\right)$ by \spcite{0FWT}{Lemma}.

We note that any $A$-algebra map $B \to \colim R_\alpha$, or $B \to R_\beta$ for a fixed $\beta$, is determined by the images of the $X_i$. This is true because the choice of the images of the $X_i$ determines a map $A[X_1,\dots,X_n]/(g_1,\dots,g_m) \to \colim R_\alpha$ (or $A[X_1,\dots,X_n]/(g_1,\dots,g_m) \to R_\beta$), which must uniquely factor through $B$ by the universal property of Henselization.

Therefore if we have two elements of $\colim \Hom(B,R_\alpha)$ which agree in $\Hom(B,\colim R_\alpha)$, we can find some large $\alpha_0$ so that the two maps both arise from $\Hom(B,R_{\alpha_0})$ and agree on all of the $X_i$, making them equal. Hence the map $\colim \Hom_A(B,R_\alpha) \to \Hom_A(B,\colim R_\alpha)$ is injective.

To show surjectivity of the map, consider an element $\eta$ of $\Hom(B,\colim R_\alpha)$. Since there are finitely many $X_i$ and the colimit is filtered, we can find some large $\alpha_0$ so that each $\eta(X_i)$ arises from some element $\eta_i \in R_{\alpha_0}$ for all $i$. This gives us an $A$-algebra map $\eta'_{\alpha_0}: A[X_1,\dots,X_n] \to R_{\alpha_0}$ defined by $\eta'_{\alpha_0}(X_i)=\eta_i$.

Since $(R_{\alpha_0},IR_{\alpha_0})$ is a Henselian pair, the universal property of Henselization means that $\eta'_{\alpha_0}$ factors through a unique $A$-algebra map $\mu_{\alpha_0}: A[X_1,\dots,X_n]^h = A\{X_1,\dots,X_n\} \to R_{\alpha_0}$. 

Furthermore, the composition of $\mu_{\alpha_0}$ with the map $R_{\alpha_0} \to \colim R_{\alpha}$ is equal to the map $$A\{X_1,\dots,X_n\} \to A\{X_1,\dots,X_n\}/(g_1,\dots,g_m) \to \colim R_{\alpha}$$ by construction. Therefore for each $j$, the image of $\mu_{\alpha_0}(g_j)$ by the map $R_{\alpha_0} \to \colim R_{\alpha}$ is $0$.

Since there are finitely many $g_j$, we can find a large $\alpha_1$ so that for each $j$, the image of $\mu_{\alpha_0}(g_j)$ by the map $R_{\alpha_0} \to R_{\alpha_1}$ is $0$. Then if we let $\eta_i' \in R_{\alpha_1}$ be the image of $\mu_{\alpha_0}(X_i)$ by the map $R_{\alpha_0} \to R_{\alpha_1}$, we can define an $A$-algebra map $\eta_{\alpha_1}: A\{X_1,\dots,X_n\}/(g_1,\dots,g_m) \simeq B \to R_{\alpha_1}$ by $\eta_{\alpha_1}(X_i)=\eta_i'$.

Considering $\eta_{\alpha_1}$ as an element of $\colim \Hom_A(B,R_\alpha)$, we see that its image in $\Hom_A(B,\colim R_\alpha)$ is $\eta$. 

Therefore the map $\colim \Hom_A(B,R_\alpha) \to \Hom_A(B,\colim R_\alpha)$ is surjective and injective, hence bijective. Hence (i) $\implies$ (ii).

To prove that (ii) $\implies$ (i), assume that the map $f: A \to B$ satisfies condition (ii); namely that for any inductive filtered system $(R_\alpha)$ of Henselian pairs over $(A,I)$, the map $\colim \Hom_A(B,R_\alpha) \to \Hom_A(B,\colim R_\alpha)$ is bijective. We must now show that $f$ is \hfp.

We can write $B$ as the (filtered) colimit of its finitely generated $A$-subalgebras, as $B = \colim C_\lambda$. Let $C_\lambda^h$ be the Henselization of $C_\lambda$ with respect to $IC_\lambda$. Then since Henselization commutes with filtered colimits \spcite{0A04}{Lemma}, we see that $\colim C_\lambda^h \simto B^h$. However, since $(B,IB)$ is already a Henselian pair, in fact $\colim C_\lambda^h \simto B$.

Since each $C_\lambda$ is finitely generated over $A$, we can write $C_\lambda \simeq A[T_1,\dots,T_r]/J_\lambda$ for some $r$ (depending on $\lambda$) and some ideal $J_\lambda \subset A[T_1,\dots,T_r]$ . Then $C_\lambda^h \simeq A\{T_1,\dots,T_r\}/J_\lambda$. 

By condition (ii), we have a bijection $\colim_\lambda \Hom_A(B, C_\lambda^h) \to \Hom_A(B,B)$. By considering the inverse image of the identity map under this bijection, we get a $\lambda_0$ and a map $B \to C_{\lambda_0}^h$ so that $B \to C_{\lambda_0}^h \to B$ is the identity map.

This means that $C_{\lambda_0}^h \to B$ is actually surjective. Since $C_{\lambda_0}^h \simeq A\{T_1,\dots,T_r\}/J$ for some $r$ and some ideal $J \subset A\{T_1,\dots,T_r\}$, we see that $B \simeq A\{T_1,\dots,T_r\}/\J$ for some ideal $\J$.

The ideal $\J$ is the colimit of its finitely generated subideals $\J_\mu$. Because $\colim$ is exact, if we write $D_\mu := A\{T_1,\dots,T_r\}/\J_\mu$, we get $B \simeq \colim D_\mu$. For each $\mu$, let $p_{\mu}: D_\mu \to B$ be the map arising from the colimit and let $q_{\mu}: A\{T_1,\dots,T_r\} \to D_\mu$ be the usual quotient map. Therefore for any $\mu,\mu'$ we have $p_{\mu}q_{\mu}=p_{\mu'}q_{\mu'}$ since both are equal to the quotient map $q:A\{T_1,\dots,T_r\} \to B$.

Once again by the assumed condition we have a bijection $\colim \Hom_A(B,D_\mu) \to \Hom_A(B,B)$, because all the $D_\mu$ are quotients of a Henselian ring and therefore Henselian themselves \spcite{09XK}{Lemma}. As above, this means we can find some $\mu_1$ with a map $u_1: B \to D_{\mu_1}$, so that $p_{\mu_1}u_1$ is the identity map on $B$.

For each $i$ from $1$ to $r$, let $t_i=q(T_i)$. Then since $p_{\mu_1}u_1=\rm{id}_B$, we have $$p_{\mu_1}u_1(t_i)=t_i=q(T_i)=p_{\mu_1}q_{\mu_1}(T_i).$$ Therefore $u_1(t_i)-q_{\mu_1}(T_i)$ lies in $\ker p_{\mu_1} = \J/\J_{\mu_1}$. There exists some finitely generated subideal of $\J/\J_{\mu_1}$ containing all of the $u_1(t_i)-q_{\mu_1}(T_i)$. 

We can thus find a large $\mu_2$ so that for all $i$, $u_1(t_i)-q_{\mu_1}(T_i) \in \J_{\mu_2}/\J_{\mu_1}$. Then if $p_{\mu_1\mu_2}: D_{\mu_1} \to D_{\mu_2}$ is the quotient map with kernel $\J_{\mu_2}/\J_{\mu_1}$, we see that we have the following commutative diagram (the dashed arrow $u_2$ will be defined later):

\begin{center}
\begin{tikzcd}[sep=4em]
&&D_{\mu_1}\arrow[dd,"p_{\mu_1\mu_2}",bend left] \arrow[d,"p_{\mu_1}" left]\\
A\{T_1,\dots,T_r\}\arrow[r,"q"]\arrow[rru,"q_{\mu_1}"]\arrow[rrd,"q_{\mu_2}" below,]&B\arrow[r,equals]\arrow[ru,"u_1"]\arrow[rd,dashed,"u_2" swap]&B\\
&&D_{\mu_2}\arrow[u,"p_{\mu_2}"]\\
\end{tikzcd}
\end{center} \vspace{-.4in}

By some diagram chasing, the condition that for all $i$, $u_1(t_i)-q_{\mu_1}(T_i) \in \J_{\mu_2}/\J_{\mu_1}$ implies that for all $i$ we have $p_{\mu_1\mu_2}u_1(t_i)=q_{\mu_2}(T_i)$ in $D_{\mu_2}$. Now let $u_2=p_{\mu_1\mu_2}u_1$. We note that because $p_{\mu_1}u_1=\rm{id}_B$, we also have $p_{\mu_2}u_2=\rm{id}_B$. Then $u_2(t_i)=q_{\mu_2}(T_i)$ for all $i$, or $u_2q(T_i)=q_{\mu_2}(T_i)$ for all $i$. This implies that $u_2q=q_{\mu_2}$ as $A$-algebra maps from $A\{T_1,\dots,T_r\}$ to $D_{\mu_2}$.

By the definition of $q$, $\J = \ker q$. Since $\ker q \subseteq \ker u_2q = \ker q_{\mu_2} = \J_{\mu_2}$, we see that $\J \subseteq \J_{\mu_2}$. However, since $\J_{\mu_2}$ is a finitely generated subideal of $\J$, this implies that $\J=\J_{\mu_2}$ is finitely generated.

Then the presentation $B \simeq A\{T_1,\dots,T_r\}/\J$ shows that $B$ is \hfp over $A$ as desired. \end{proof}

\subsection{Zariski-locality}

In Section~\ref{sub:basichfp}, we discussed \hfp maps of Henselian pairs and defined \hlfp maps of Henselian schemes. The definition of the latter was given with respect to a choice of open covering of the Henselian schemes in question. In this section we will show that being \hfp is Zariski-local, which will imply that being \hlfp can be checked using any affine open covering of a Henselian scheme.

As part of showing the condition of being \hfp is Zariski-local, we will need a Henselian version of a standard fact for schemes concerning distinguished affine opens:

\begin{proposition}\label{prop:acl} If $\sph(A,I)$ and $\sph(B,J)$ are affine opens in a Henselian scheme $X$, then $\sph(A,I) \cap \sph(B,J)$ is a union of open sets which are distinguished affine opens of both $\sph(A,I)$ and $\sph(B,J)$.
\end{proposition}
\begin{proof} The underlying topological space of an affine Henselian scheme $\sph(A,I)$ is $\spec(A/I)$, and a distinguished affine open subset $\spec((A/I)_{\ol{f}})$ with $\ol{f} \in A/I$ corresponds to a distinguished affine open in $\sph(A,I)$ with underlying topological space $\spec((A/I)_{\ol{f}})=D(f) \cap V(I)$ (where $f$ is any lift of $\ol{f}$) that can be described as $\sph(A_f^h,IA_f^h)$. 

Then if $\mc{I}$ is the radical ideal sheaf of $\O_X$ comprising sections which vanish at the residue fields of every point, we are reduced to the scheme case applied to the scheme $Z=(X,\O_X/\mc{I})$ (see Remark~\ref{rmk:getususch}). 

This gives us a cover of $\spec(A/I) \cap \spec(B/J) \subset Z$ by open sets which are distinguished affine opens of both $\spec(A/I)$ and $\spec(B,J)$. This can then be lifted to give the desired cover of $\sph(A,I) \cap \sph(B,J)$ by distinguished affine opens.\end{proof}

\begin{proposition}\label{prop:affglue} 
Consider a map of rings $f: A \to B$ where $(A,I)$ and $(B,IB)$ are Henselian pairs, and a finite cover of $\sph(B)$ by affine opens $U_i = \sph(B_i)$, where $B_i$ is an $A$-algebra such that $(B_i,IB_i)$ is Henselian. Then if each $B_i$ is \hfp over $A$, so is $B$.
\end{proposition}
\begin{proof} We will make use of the functorial criterion of Proposition~\ref{prop:groth}.

Each $U_i$ is covered by finitely many open sets $U_{i,j}$ which are distinguished affine opens of both $U_i=\sph(B_i)$ and of $\sph(B)$. Then by Example~\ref{ex:disthfp} and Proposition~\ref{prop:hfpcomp}, we see (since $B_i$ is \hfp over $A$ and $U_{i,j}$ is \hfp over $B_i$) that $U_{i,j}$ is a distinguished affine open of $\sph(B)$ which is \hfp over $A$.

Therefore we have reduced to the case where for each $i$ there exists some $f_i \in B$ so that $B_i=B_{f_i}^h$, and $B_{f_i}^h$ is \hfp over $A$. 

Because the $\sph(B_{f_i}^h)$ cover $\sph(B)$, we know that we can cover $\spec(B/IB)$ by affine opens $\spec((B/IB)_{\ol{f_i}})$ for $\ol{f_i}$ the image of $f_i$ in $B/IB$. Then the ideal of $B/IB$ generated by the $\ol{f_i}$ is the unit ideal. Because $(B,IB)$ is Henselian, the $f_i$ generate the unit ideal of $B$ as well.

By Proposition~\ref{prop:groth}, it suffices to show that for any inductive filtered system of Henselian pairs $(R_\alpha)$ over $A$, the map $\colim \Hom_A(B,R_\alpha) \to \Hom_A(B,\colim R_\alpha)$ is bijective.

We will first show that this map is surjective. Let $R=\colim R_\alpha$ and fix some $\vp \in \Hom_A(B,R)$. For each $i$ we get a map $\vp_{f_i}^h: B_{f_i}^h \to R_{\vp(f_i)}^h$ from the map $\vp$. 

Since there are finitely many $i$, we can find some large $\alpha_0$ so that for all $i$, $\vp(f_i) \in R$ arises from some $r_i \in R_{\alpha_0}$. Then since the $f_i$ generate the unit ideal of $B$, we can (possibly by increasing $\alpha_0$ further) assume that the $r_i \in R_{\alpha_0}$ generate the unit ideal.

Now because $B_{f_i}^h$ is \hfp over $A$ for all $i$, we can apply Proposition~\ref{prop:groth} to the $\vp_{f_i}^h$ and increase $\alpha_0$ so that we have maps $\phi_{f_i}^h: B_{f_i}^h \to (R_{\alpha_0})_{r_i}^h$ such that all of the $\vp_{f_i}^h$ arise from the corresponding maps $\phi_{f_i}^h$, and so that for all $i,j$ we have $\phi_{f_i}^h(f_j)=r_j$. Therefore we can write $(R_{\alpha_0})_{f_i}^h$ for $(R_{\alpha_0})_{r_i}^h$ without ambiguity.

We now consider the situation geometrically. We get a diagram for each $i$

\begin{center}
\begin{tikzcd}[column sep=1em]
\sph(B) &&\arrow[ll] \sph(\colim R_\alpha)\\
\sph(B_{f_i}^h) \arrow[u] && \sph(\colim(R_{\alpha})_{f_i}^h) \arrow[u]\arrow["\vp_{f_i}^h",ll] \arrow[dl]\\
&\sph((R_{\alpha_0})_{f_i}^h) \arrow[lu,"\phi_{f_i}^h"]&\\
\end{tikzcd}
\end{center} \vspace{-.25in}

such that the top square and the bottom triangle both commute. Furthermore, for each pair $i,j$ we get Cartesian diagrams for the overlaps

\begin{center}
\begin{tikzcd}
\sph(B_{f_i}^h) & \sph((R_{\alpha_0})_{f_i}^h) \arrow["\phi_{f_i}^h",l] & \sph(B_{f_j}^h) & \sph((R_{\alpha_0})_{f_j}^h) \arrow["\phi_{f_j}^h",l]\\
\sph(B_{f_if_j}^h)\arrow[u]& \sph((R_{\alpha_0})_{f_if_j}^h) \arrow["\psi_{ij}",l]\arrow[u] & \sph(B_{f_if_j}^h)\arrow[u]& \sph((R_{\alpha_0})_{f_if_j}^h) \arrow["\psi_{ji}",l]\arrow[u]
\end{tikzcd}
\end{center} \vspace{-.1in}

in which both squares commute. We wish to show that $\psi_{ij}=\psi_{ji}$.  In fact they become equal to $\phi_{f_if_j}^h: \sph(\colim (R_\alpha)_{f_if_j}^h) \to \sph(B_{f_if_j}^h)$ after passing to the colimit since $\phi_{f_i}^h$ and $\phi_{f_j}^h$ must agree on the overlap as both arise from $\vp$. Since all of the $B_{f_i}^h$ are \hfp over $A$, the same holds for each $B_{f_if_j}^h$. Hence by applying Proposition~\ref{prop:groth}, we can increase $\alpha_0$ to assume that $\psi_{ji}$ and $\psi_{ij}$ are equal.

Because the $r_i$ generate the unit ideal in $R_{\alpha_0}$, $\sph(R_{\alpha_0})$ is covered by the $\sph((R_{\alpha_0})_{f_i}^h)$. Then the maps $ \sph((R_{\alpha_0})_{f_i}^h) \to \sph(B_{f_i}^h)$ agree as maps into $\sph(B)$ by the argument above, giving us a map $\sph(R_{\alpha_0}) \to \sph(B) $. 

After passing to the limit, this map $\sph(R_{\alpha_0}) \to \sph(B) $ agrees with the original $\vp$ on each of the affine opens $\sph(R_{f_i}^h)$ of the cover by construction. Therefore $\sph(R_{\alpha_0}) \to \sph(B) $ gives rise to the original $\vp$ on all of $\sph(R)$. This shows that $\colim \Hom_A(B,R_\alpha) \to \Hom_A(B,\colim R_\alpha)$ hits $\vp$ and so is surjective.

We now will show that the map $\colim \Hom_A(B,R_\alpha) \to \Hom_A(B,\colim R_\alpha)$ is injective. 

If we have two maps $B \to R_{\alpha_1}, B \to R_{\alpha_2}$ we can assume without loss of generality that $\alpha_1=\alpha_2$ since the colimit is filtered. 

Now assume we have two maps $\vp_1,\vp_2: B \to R_{\beta}$ which give rise to the same map $\vp: B \to R$. By increasing $\beta$, we can assume that for all $i$, $\vp_1(f_i)=\vp_2(f_i)$.

Now for each $i$ let $\vp_{1,f_i},\vp_{2,f_i}: B_{f_i}^h \to (R_{\beta})_{f_i}^h$ be the two maps arising from $\vp_1,\vp_2$. Writing $(R_{\beta})_{f_i}^h$ is unambiguous since $\vp_1(f_i)=\vp_2(f_i)$.

Since we assumed $B_{f_i}^h$ was \hfp over $A$, by Proposition~\ref{prop:groth} the map  $$\colim \Hom_A(B_{f_i}^h,(R_\alpha)_{f_i}^h) \to \Hom_A(B_{f_i}^h,\colim (R_\alpha)_{f_i}^h)$$ is injective. After passing to the colimit $\colim (R_\alpha)_{f_i}^h=R_{f_i}^h$, the maps $\vp_{1,f_i}$ and $\vp_{2,f_i}$ become equal, so after increasing $\beta$ again we can assume they are equal as maps $B_{f_i}^h \to (R_{\beta})_{f_i}^h$. 

Increasing $\beta$ once again, we can assume the $f_i$ generate the unit ideal of $R_\beta$. Since $\vp_{1,f_i}=\vp_{2,f_i}$ for all $i$, we see that $\vp_1=\vp_2$ as maps $B \to R_\beta$. Hence $\colim \Hom_A(B,R_\alpha) \to \Hom_A(B,\colim R_\alpha)$ is injective as we wished to show.

Therefore since $\colim \Hom_A(B,R_\alpha) \to \Hom_A(B,\colim R_\alpha)$ is both surjective and injective, $B$ is \hfp over $A$ by Proposition~\ref{prop:groth}. \end{proof}

Now we can show that the condition of being \hlfp may be checked Zariski-locally.

\begin{proposition}\label{prop:everyaff} Consider a map of Henselian schemes $f: (X,\O_X) \to (Y,\O_Y)$ which is \hlfp. Then for any affine opens $U=\sph(C',\mc{I}') \subset X, V = \sph(C,\mc{I}) \subset Y$ with $f(U) \subset V$, the corresponding map of rings $C \to C'$ is \hfp.
\end{proposition}
\begin{proof}Let $x \in U \subset X$ be an arbitrary point. Then by the definition of being \hlfp we can find neighborhoods $V_{f(x)} \simeq \sph(A,I) \subset V$ and $U_x \simeq \sph(B,IB) \subset U$, such that $f(U_x) \subset V_{f(x)}$ and $B \simeq A\{X_1,\dots,X_n\}/(g_1,\dots,g_m)$. 

By Proposition~\ref{prop:acl}, we can choose a distinguished affine open $V_{f(x)}' \subseteq V_{f(x)}$ which contains $f(x)$ and is also a distinguished affine open of $V$. By Example~\ref{ex:disthfp}, we see that $V_{f(x)}'$ is \hfp over $V$. 

 If $V_{f(x)}'=\sph(A_a^h,IA_a^h)$ we then see that $U_x':=f^{-1}(V_{f(x)}') \cap U_x$ is also a distinguished affine open of $U_x$ by Example~\ref{ex:disthfp} and Lemma~\ref{lem:polyonpoly}. 
 
Then $U_x'$ is \hfp over $V_{f(x)}'$, hence over $V=\sph(C)$ (Proposition~\ref{prop:hfpcomp}). Hence for any $x \in U$ we can find an affine neighborhood $U_x'$ of $x$ which is \hfp over $V$. Applying Proposition~\ref{prop:affglue}, we see that $C'$ is \hfp over $C$ as we desired to show.\end{proof}

\section{\'Etale and smooth morphisms}\label{sec:etsm}

Now that we have compatible concepts of \hfp and \hlfp maps, we can define ``Henselian \'etale'' and ``Henselian smooth'' morphisms in this section. 

This was discussed in \cite{cox1, kpr}, but the definition of ``Henselian \'etale'' given there assumed that flatness is a consequence of a formal \Et lifting criterion; in Remark~\ref{rmk:kprwrong} we will show this is not the case. Instead we will define ``Henselian smooth'' or ``Henselian \Et'' maps as the Henselization of a smooth or \Et map in the affine case, and show that this is equivalent to being \hlfp, flat, and satisfying an appropriate lifting criterion.

\subsection{Preliminaries}

We will often wish to descend certain properties from the Henselization of some algebra $R$ at an ideal $I$ to an \Et (hence finitely presented) $R$-algebra by showing that the property in question holds at all the points of $V(I)$. In Corollary~\ref{cor:niceprops} of the \hyperref[sec:appendix]{Appendix}, we show this can be done for the properties of being quasi-finite, flat, or smooth, recorded here:

\begin{restatable*}{corollary}{corniceprops}\label{cor:niceprops}
Let $(A,I)$ be a Henselian pair and $R$ a finitely presented $A$-algebra such that the map $\spec(R) \to \spec(A)$ is quasi-finite (resp. flat, resp. smooth) at each prime in $V(IR) \subset \spec(R)$. Then we can find some $R'$ \Et over $R$ such that $R'/IR' \simeq R/IR$ and the map $A \to R'$ is quasi-finite (resp. flat, resp. smooth).
\end{restatable*}

Given a map of Henselian pairs $(A,I) \to (B,IB)$, it will be useful to lift certain finiteness properties of the map $A/I \to B/IB$ to the map $A \to B$. This will let us relate $A$-algebras satisfying an infinitesimal lifting condition to the Henselizations of smooth or \Et $A$-algebras by comparing them modulo $I$.

\begin{lemma}\label{lem:betterfin} Consider a map of rings $f: A \to B$ where $(A,I)$ and $(B,IB)$ are Henselian pairs. If $f$ is \hfp and $\ol{f}: A/I \to B/IB$ is module-finite, then $f$ is module-finite and finitely presented.
\end{lemma}
\begin{proof} If $f$ is \hfp, then by Lemma~\ref{lem:hfpishoffp} we can find a finitely presented $A$-algebra $R$ such that $B \simeq R^h$ (the $I$-adic Henselization of $R$). Let $A_0 := A/I$; define $B_0,R_0$ similarly. Since $B \simeq R^h$, we have $B_0 \simeq R_0$. By hypothesis, $B_0$ is $A_0$-finite.

Because $R_0$ is $A_0$-finite (hence quasi-finite), the morphism $\spec(R) \nobreak\to\nobreak \spec(A)$ is quasi-finite at each point of $V(IR)$. Therefore we can apply Corollary~\ref{cor:niceprops} to find some $R_{qf}$ which is \Et over $R$, with $R/IR \simeq R_{qf}/IR_{qf}$, so that $R_{qf}$ is quasi-finite over $A$. We now replace $R$ with $R_{qf}$. 

The map $R \to B$ is flat by \spcite{0AGU}{Lemma}. Since we have arranged that $R$ is quasi-finite over $A$, and the map $\spec(R) \to \spec(A)$ is affine, hence separated, we can apply Zariski's Main Theorem \spcite{05K0}{Lemma} to find a finite and finitely presented $A$-algebra $S$ with an open immersion $\spec(R) \hookrightarrow \spec(S)$ over $A$. As above, write $S_0$ for $S/IS$. Since $S$ is finite, hence integral, over $A$, the pair $(S,IS)$ is Henselian. 

Since $R_0$ is finite over $A_0$, it is also finite over $S_0$. Therefore $\spec(R_0) \hookrightarrow \spec(S_0)$ is both finite and an open immersion. Since a finite map is closed, the image of $\spec(R_0)$ is clopen in $\spec(S_0)$. Therefore we can write $S_0$ as a product of rings $R_0 \times S_0'$. Since $(S,IS)$ is Henselian, we can use the equivalent definitions of Henselian in the \hyperref[sec:appendix]{Appendix} (Lemma~\ref{lem:sphenspair}) to get a ring decomposition $S = \widetilde{R} \times S'$, where $\widetilde{R}/I\widetilde{R} \simeq R_0, S'/IS' \simeq S_0'$, via the bijection of the idempotents in $S$ and in $S_0$.

Now from the map $S \to R$ we get a ring decomposition $R = R_1 \times R_2$ with $R_2/IR_2=0$ corresponding to the images of the idempotents giving the decomposition of $S$. Because $R_2/IR_2=0$, we have $R_1/IR_1 \simeq R/IR$, so $R_1 \ne 0$. Then since $R \to R_1$ is a localization at an idempotent, it is a nonzero \Et map; furthermore, it is an isomorphism modulo $I$, so $R^h \simeq R_1^h \simeq B$ (Proposition~\ref{prop:hensetexp}).

Then the open immersion $\spec(R) \hookrightarrow \spec(S)$ gives us an open immersion $\spec(R_1) \hookrightarrow \spec(\widetilde{R})$ which is an isomorphism along $I$, with $\widetilde{R}$ finite and finitely presented over $A$ since $S$ is. Since an open immersion is \Et and $\widetilde{R} \to R_1$ is an isomorphism modulo $I$, we have isomorphisms $\widetilde{R}^h \simeq R_1^h \simeq B$ (Proposition~\ref{prop:hensetexp}). However, $\widetilde{R}$ is finite, hence integral, over $A$, so $(\widetilde{R},I\widetilde{R})$ is a Henselian pair by \spcite{09XK}{Lemma}. Therefore $B \simeq \widetilde{R}$, so $B$ is finite and finitely presented over $A$.\end{proof}

\begin{lemma}\label{lem:surj} Consider a map of rings $f: A \to B$ where $(A,I)$ and $(B,IB)$ are Henselian pairs. If $f$ is \hfp and $\ol{f}: A/I \to B/IB$ is surjective, then $f$ is surjective.
\end{lemma}
\begin{proof} By Lemma~\ref{lem:betterfin}, we know that $f$ is module-finite and finitely presented. 

Let $C$ be the cokernel of $f$ considered as an $A$-module map. It is clearly finitely generated over $A$ as a module. Since $\ol{f}$ is surjective, we see that $C/IC = 0$. Because $(A,I)$ is a Henselian pair, $I \subseteq \Jac(A)$. Then by Nakayama's lemma we have $C=0$; hence $f$ is surjective.\end{proof}

\subsection{Defining \hsmooth and \hetale morphisms}

A concrete way to define a ``Henselian smooth'' map of Henselian pairs is as the Henselization of a smooth map:

\begin{definition}\label{def:hsmet}
A map of pairs $(A,I) \to (B,J)$  is {\bf Henselian smooth} or {\bf \hsmooth} if there exists a smooth $A$-algebra $R$ such that the Henselization $R^h$ of $R$ along $IR$ is $A$-isomorphic to $B$, with $\sqrt{IB}=\sqrt{J}$.

Similarly, a map of pairs $(A,I) \to (B,J)$  is {\bf Henselian \'etale} or {\bf \hetale} if there exists a \Et $A$-algebra $R$ such that the Henselization $R^h$ of $R$ along $IR$ is $A$-isomorphic to $B$, with $\sqrt{IB}=\sqrt{J}$.
\end{definition}
\begin{remark}\label{rmk:hsmetbasechange}
By a similar argument to that in Remark~\ref{rmk:hlfpbasechange} (replacing the finitely presented $A$-algebra with an \Et or smooth $A$-algebra), the property of being \hetale or \hsmooth is stable under base change.\end{remark}

These definitions are ``affine-global''; it is not obvious they can be checked Zariski-locally. Therefore we now define infinitesimal lifting conditions for maps of Henselian pairs that are analogous to the notions of formal {\Et}ness or smoothness for general rings; these definitions will be easier to work with in some respects. In particular, we will show in Propositions~\ref{prop:flathenset}~and~\ref{prop:flathenssm} that Definition~\ref{def:hsmet} in the \hfp and flat case is equivalent to a suitable lifting condition. This will yield that Definition~\ref{def:hsmet} is Zariski-local.

\begin{definition}\label{def:hfsmet}
We say that a map of Henselian pairs $(A,I) \to (B,IB)$ is {\bf Henselian formally smooth} or {\bf \hfsmooth} (resp. {\bf Henselian formally \Et} or {\bf \hfetale}) if for any diagram of $A$-homomorphisms

\begin{center}
\begin{tikzcd}
E \arrow[r,"p",twoheadrightarrow]& C\\
B\arrow[u,"v",dashed] \arrow[ru,"u"]\\
\end{tikzcd}
\end{center} \vspace{-.25in}

where $I^NE=0, I^NC=0$ for some $N \ge 1$ and $p$ is surjective with $(\ker p)^M=0$ for some integer $M > 0$, there exists a (unique, for \hfetale) dashed arrow $v$ over $A$ making the diagram commute. 
\end{definition}

It is clear that if a map is \hfetale, then it is \hfsmooth. 

We will show that for an \hfp morphism $\phi: (A,I) \to (B,IB)$ of Henselian pairs, the lifting condition of Definition~\ref{def:hfsmet} is equivalent to the condition that for all $r \ge 1$, the morphism $\phi_r: A/I^r \to B/I^rB$ is \Et (or smooth). The latter is often an easier property to check. 

In particular, because localization commutes with these quotients by powers of $I$, this equivalence yields that the properties of being \hfetale or \hfsmooth are Zariski-local in the \hfp case.

\begin{lemma}\label{lem:finet} Given an \hfp map of Henselian pairs $\phi:(A,I) \to (B,IB)$, the following are equivalent:

\begin{enumerate}[(i)]
\item The map $\phi$ is \hfetale;
\item For all $r \ge 1$, the map $A/I^r \to B/I^rB$ is \Et.
\end{enumerate}

\end{lemma}
\begin{proof} 

We first show that (i) $\implies$ (ii). Assume $\phi$ is \hfetale and fix some $r \ge 1$. 

Consider $\ol{f}_r: A/I^r \to B/I^rB$ and a map $q: R \to S$ of $A/I^r$-algebras such that $q$ is surjective with some power of $\ker q$ being $0$. If we have a homomorphism $g: B/I^rB \to S$, then we have a commutative diagram of $A$-homomorphisms:

\begin{center}
\begin{tikzcd}
R \arrow[r,"q",twoheadrightarrow] & S\\
B \arrow[ru,"\widetilde{g}"] \arrow[u,"\widetilde{h}",dashed] \arrow[r,twoheadrightarrow]& B/I^rB \arrow[u,"g"] \\
\end{tikzcd}
\end{center} \vspace{-.25in}

Therefore, by condition (i), there exists a unique dashed arrow $\widetilde{h}: B \to R$ over $A$ making the diagram commute. Then since $R$ is an $A/I^r$-algebra, we must have $\widetilde{h}(I^rB)=0$; this gives us a map $h: B/I^rB \to R$ such that $qh=g$. 

If there existed a second $A/I^r$-homomorphism $h': B/I^rB \to R$ with $qh'=g$, then the composition of $h'$ with the quotient map $B \to B/I^rB$ would have to equal $\widetilde{h}$ since $\phi$ is \hfetale. However, that would mean that $h=h'$ as $B \to B/I^rB$ is surjective. Therefore $h$ is unique.

Therefore for any map $q: R \to S$ of $A/I^r$-algebras which is surjective with nilpotent kernel, any $A$-homomorphism $g: B/I^rB \to S$ has a unique lift $h: B/I^rB \to R$ over $A$ with $qh=g$. Hence $\ol{f}_r$ satisfies the lifting property of \cite[Definition 19.10.2]{ega41}, and $\ol{f}_r$ is finitely presented since $f$ is \hfp. Therefore by \cite[Theorem 17.6.1]{ega44}, $\ol{f}_r$ is \Et. 

Therefore (i) $\implies$ (ii). We now will show that (ii) $\implies$ (i); begin by assuming that for all $r$ the map $\ol{f}_r: A/I^r \to B/I^rB$ is \Et.

Now consider an arbitrary surjective homomorphism of discrete $A$-algebras, $p: E \to C$, with some power of $\ker p$ equal to $0$, and an $A$-homomorphism $u: B \to C$. Note that $E,C$ are  $A/I^n$ algebras for some $n \ge 1$, so $u$ must factor through $B/I^nB$. Therefore we have a diagram of $A/I^n$-algebras:

\begin{center}
\begin{tikzcd}
E \arrow[r,"p",twoheadrightarrow] & C\\
B/I^nB \arrow[ru,"\ol{u}"] \arrow[u,"\ol{v}",dashed] \\
\end{tikzcd}
\end{center} \vspace{-.25in}

Recall that by condition (ii) $\ol{f}_n: A/I^n \to B/I^nB$ is \Et. So because $p: E \to C$ is a surjective homomorphism of $A/I^n$-algebras with nilpotent kernel, there exists a unique dashed arrow $\ol{v}$ over $A/I^n$ making the diagram commute \cite[Theorem 17.6.1]{ega44}. It is then evident that the lift $v: B \to E$ of $\ol{v}$ is a $A$-homomorphism with $pv=u$. It remains only to show that $v$ is unique. 

Consider a second $w: B \to E$ with $pw=u$, so $w$ factors through $B/I^nB$. Both $w$ and $v$ factor through $B/I^n B$, so since (by {\Et}ness over $A/I^n$) both factor through the same map $B/I^nB \to E$, they are equal as maps $B \to E$. Therefore the lifting property is satisfied and $\phi$ is \hfetale.\end{proof}

\begin{lemma}\label{lem:finsm} Given an \hfp map of Henselian pairs $\phi:(A,I) \to (B,IB)$, the following are equivalent:

\begin{enumerate}[(i)]
\item The map $\phi$ is \hfsmooth;
\item For all $r \ge 1$, the map $A/I^r \to B/I^rB$ is smooth.
\end{enumerate}

\end{lemma}
\begin{proof}
The proof is the same as that of Lemma~\ref{lem:finet}, just without any uniqueness assumptions. The relevant lifting property is \cite[Definition 19.3.1]{ega41}, which proves smoothness by \cite[Theorem 17.5.1]{ega44}.\end{proof}

We can now relate the infinitesimal lifting conditions  of being \hfsmooth or \hfetale  to the more concrete notions of \hsmooth  or \hetale algebras.

\begin{proposition}\label{prop:hetisformal} For a map of rings $\phi: A \to B$ with $(A,I)$ and $(B,IB)$ both Henselian pairs (for $I$ an ideal of $A$), if $\phi$ is \hetale (resp. \hsmooth) then it is \hfetale (resp. \hfsmooth).
\end{proposition}
\begin{proof} 
Assume $\phi$ is \hetale. Therefore $B$ is the Henselization of an \Et $A$-algebra $R$, so for all $r$ we have $B/I^rB \simeq R/I^rR$, which is \Et over $A/I^r$ because the base change of an \Et map is \Et. Furthermore since $R$ is \Et, hence finitely presented, over $A$, the map $\phi$ is \hfp. Hence by Lemma~\ref{lem:finet}, $\phi$ is \hfetale. 

The proof for $\phi$ \hsmooth is similar, using Lemma~\ref{lem:finsm}.\end{proof}

The converse of Proposition~\ref{prop:hetisformal} is not always true, but we can show that an \hfetale (\hfsmooth) algebra is the quotient of an \hetale (\hsmooth) algebra by a suitable ideal.

\begin{theorem}\label{thm:hetale} Given an \hfp map of rings $\phi: A \to B$ and an ideal $I \subset A$ such that both pairs $(A,I)$ and $(B,IB)$ are Henselian, if $\phi$ is \hfetale (resp. \hfsmooth), then there exists an \Et (resp. smooth) $A$-algebra $R$ with $B \simeq R^h/J$ for $R^h$ the Henselization of $R$ along $IR$ and $J$ a finitely generated ideal of $R^h$ with $J \subseteq \bigcap_{r \ge 1} I^rR^h$. \end{theorem}\begin{remark}\label{rmk:quothet} Conversely, any such quotient $B=R^h/J$ is \hfetale (resp. \hfsmooth), since $B/I^n B \simeq R^h/I^n R^h \simeq R/I^nR$ for all $n \ge 0$ \spcite{0AGU}{Lemma}. A basic example of this with nonzero $J$ can be found in Example~\ref{ex:hensnotsep} of the \hyperref[sec:appendix]{Appendix}, where $R = A=\mc{C}^\infty_0(\R,\R)$ and $J$ is the principal ideal generated by $e^{-1/x^2}$.
\end{remark}
\begin{proof} The ring $B/IB$ is \Et (resp. smooth) over $A/I$ by Lemmas~\ref{lem:finet}~and~\ref{lem:finsm}. We can obtain a smooth $A$-algebra $R$ such that $R/IR \simeq B/IB$  (\cite[Theorem 1.3.1]{arabia}, \spcite{07M8}{Proposition}).

In the \hetale case, we can construct $R$ more explicitly to be \Et over $A$: we have a standard smooth presentation of $B/IB$ over $A/I$ by \spcite{00U9}{Lemma} which can be lifted easily to a standard smooth and \Et $A$-algebra $R$. 

Now we will find a map $R^h \to B$, where $R^h$ is the Henselization of the pair $(R,IR)$, which lifts the isomorphism $R^h/IR^h \simeq R/IR \simeq B/IB$. Since $B$ is Henselian, it suffices to find an $A$-algebra map $R \to B$ lifting the isomorphism $R/IR \simeq B/IB$.

We consider the tensor product $R \tens_A B$, which is \Et (resp. smooth) over $B$. The map $R \to B$ which we desire will factor through this tensor product.

We can calculate $(R \tens_A B)/I(R \tens_A B)$ as $$A/I \tens_A (R \tens_A B)  = (R \tens_A  A/I) \tens_{A/I} (A/I \tens_A B) = R/IR \tens_{A/I} B/IB = B/IB \tens_{A/I} B/IB,$$

so if we compose the quotient map with the multiplication map, we get a map $R \tens_A B \to B/IB$.

In the \Et case, since $B$ is Henselian, we get a lift $R \tens_A B \to B$, which gives us an $A$-algebra map $R \to B$. This map $R \to B$ is an isomorphism modulo $I$ by construction. Furthermore, by the universal property of Henselization, it factors through $R^h$, and the map $R^h \to B$ is also an isomorphism modulo $I$.

In the more general smooth case, we can find an \Et $B$-algebra $B'$ with $B/IB \simeq B'/IB'$ and a map $R \tens_A B \to B'$ lifting the map $R \tens_A B \to B/IB$ by \spcite{07M7}{Lemma}. However, by the universal property of Henselization and Proposition~\ref{prop:hensetexp}, $(B')^h \simeq B^h \simeq B$. Therefore the map $R \tens_A B \to B' \to (B')^h \simto B$ also lifts $R \tens_A B \to B/IB$. Furthermore, since $B$ is Henselian, the map $R \to R \tens_A B \to B$ factors through $R^h$, and the map $R^h \to B$ is an isomorphism modulo $I$ by construction.

Therefore in either case we have a map $R^h \to B$ which lifts the isomorphism $R^h/IR^h \simeq B/IB$.

The ring $R^h$ is \hfp over $A$ since $R$ is finitely presented over $A$. Since $B$ is also \hfp over $A$, the map $R^h \to B$ must also be \hfp by Proposition~\ref{prop:hlfphlfp}. By applying Lemma~\ref{lem:betterfin} and Lemma~\ref{lem:surj}, the map $R^h \to B$ is finitely presented in the usual sense, surjective, and an isomorphism modulo $I$.

For each $r \ge 1$, the morphism $\phi_r: R^h/I^rR^h \to B/I^rB$ is finitely presented, surjective, and an isomorphism modulo $I$. Since $R^h/I^rR^h\simeq R/I^rR$ is $A/I^r$-smooth, and $B/I^rB$ is as well by Lemma~\ref{lem:finsm}, both are flat over $A/I^r$. Therefore $R^h/I^rR^h \twoheadrightarrow B/I^rB$ is also flat (by fibral flatness \cite[Theorem 11.3.10]{ega43}), hence faithfully flat and in fact an isomorphism.

Since $R^h \to B$ is surjective by Lemma~\ref{lem:surj} and finitely presented, we can write $B \simeq R^h/J$ for $J$ a finitely generated ideal of $R^h$. Since $R^h \to B$ is an isomorphism modulo every positive power of $I$, we see that  for all $r \ge 1$ we must have $J \subset I^rR^h$. Hence $J \subset \bigcap_{r \ge 1} I^rR^h$.\end{proof}

Essentially, Theorem~\ref{thm:hetale} says that an \hfetale $A$-algebra $B$ which is \hfp is precisely the quotient of an \hetale algebra $R^h$ by a finitely generated ideal which is contained in $\bigcap_{r \ge 1} I^rR^h$. Similarly, an \hfsmooth $A$-algebra which is \hfp is precisely the quotient of an \hsmooth algebra $R^h$ by a finitely generated ideal contained in $\bigcap_{r \ge 1} I^rR^h$.

\begin{remark}\label{rmk:noethhetale} In the setting of Theorem~\ref{thm:hetale} with Noetherian $A$, the flatness issue does not arise. Let $A,B,R,J$ be as in Theorem~\ref{thm:hetale}. If $A$ is Noetherian, then $R^h$ is Noetherian by \spcite{0AGV}{Lemma}. Since $(R^h,IR^h)$ is Henselian, we have $IR^h \subseteq \Jac(R^h)$. The Krull intersection theorem then yields $\bigcap_{r \ge 1} I^rR^h = 0$, so $J=0$ and $B \simeq R^h$. 

Therefore if $A$ is Noetherian, the notions of \hfetale and \hetale are the same in the \hfp case. Similarly, for $A$ Noetherian, an \hfp $A$-algebra $B$ is \hsmooth over $A$ if and only if it is \hfsmooth over $A$. \end{remark}

We cannot avoid the interference of $J$ in Theorem~\ref{thm:hetale} in the general case, since for a Henselian pair $(C,\mf{I})$ it is not necessarily the case that $\bigcap_{r \ge 1} \mf{I}^r = 0$ -- i.e., Henselian pairs are not necessarily adically separated. An example of this phenomenon is given in Example~\ref{ex:hensnotsep} of the \hyperref[sec:appendix]{Appendix} (as we noted in Remark~\ref{rmk:quothet}). 

Furthermore, such separatedness is not inherited from that of a base ring beyond the Noetherian case (see Remark~\ref{rmk:noethhetale}). For an example of an \hetale algebra $B$ over an adically separated Henselian pair $(A,I)$ such that $B$ is {\it not} $I$-adically separated, an example due to Gabber can be found in \cite[Example 5.2.13]{sheelathesis}; it has a lengthy formulation and thus we omit it here.

More generally, we can show that the condition of being \hetale is equivalent to the condition of being \hfetale, \hfp, and flat:

\begin{proposition}\label{prop:flathenset} Let $\phi: (A,I) \to (B,IB)$ be a map of Henselian pairs. Then $\phi$ is \hfp, \hfetale, and flat if and only if it is \hetale.
\end{proposition}
\begin{proof} 

First assume that $\phi$ is \hetale (hence \hfp by definition). Then by Proposition~\ref{prop:hetisformal}, $\phi$ is \hfetale. By the definition of \hetale, $B \simeq R^h$ for $R$ \Et over $A$. Since $R \to R^h \simeq B$ is flat and \Et maps are flat, we also know that $\phi$ is flat. Therefore if $\phi$ is \hetale, it is \hfetale and flat.

To show the other direction of the implication, assume that $\phi$ is \hfp, \hfetale, and flat. By Theorem~\ref{thm:hetale}, we can find an \Et $A$-algebra $R$ with $B \simeq R^h/J$ for $R^h$ the Henselization of $R$ along $IR$ and $J$ a finitely generated ideal of $R^h$ with $J \subseteq \bigcap_{r \ge 1} I^rR^h$. Consider the diagram

\begin{center}
\begin{tikzcd}
 &R^h\arrow[r]&B\\
R\arrow[ur,"\text{flat}" sloped]\arrow[r,"\text{flat}"] &R \tens_A B \arrow[dashed]{ur}&\\
A\arrow[u,"\text{\Et}"]\arrow[r,"\text{flat}"]  &B\arrow[u,"\text{\Et}"]\arrow[equal]{uur}&\\
\end{tikzcd}
\end{center} \vspace{-.25in}

The dotted arrow exists by the universal property of the tensor product; furthermore, it is \Et since it is a map between \Et $B$-algebras (one being $B$ itself). Then since $R \tens_A B \to B$ is \Et and $R \to R \tens_A B$ is flat, $R \to B$ is flat.

We recall $R^h$ is the colimit of \Et $R$-algebras $S$ with $S/IS \simeq R/IR$. For each $S$ we have a map $S \to B$ arising from the universal property of Henselization, with $R^h \to B$ being the colimit of these maps. Since $R \to B$ is flat and $R \to S$ is \Et, we deduce that the map $S \to B$ is flat for each $S$.  
Therefore $R^h \to B$ is flat since it is the colimit of the flat maps $S \to B$ \spcite{05UT}{Lemma}. 

Since $B \simeq R^h/J$, the (finitely presented) map $\spec(B) \to \spec(R^h)$ is both flat and a closed immersion. Therefore $V(J) \subset \spec(R^h)$, which is homeomorphic to $\spec(B)$, is both open and closed (since finitely presented flat maps are open by \spcite{01UA}{Lemma}). Hence we get a decomposition of rings $R^h = B \times B'$; since $R^h/IR^h \simeq B/IB$, necessarily $B'/IB'=0$. However, since the pair $(R^h,IR^h)$ is Henselian, the map $R^h \to R^h/IR^h$ induces a bijection on idempotents (Lemma~\ref{lem:sphenspair}), so $B'/IB'=0$ implies that $B'=0$ and $R^h \simeq B$. \end{proof}

Similarly, the condition of being \hsmooth is equivalent to the condition of being \hfp, \hfsmooth, and flat. To prove this, we will first show that for a Henselian pair $(A,I)$, an \hfp and flat $A$-algebra is the $I$-adic Henselization of a flat and finitely presented $A$-algebra.

\begin{lemma}\label{lem:flathfp} Given a map of rings $\phi: A \to B$ and an ideal $I \subset A$ such that both pairs $(A,I)$ and $(B,IB)$ are Henselian, if $\phi$ is \hfp and flat then we can find a finitely presented $A$-algebra $S$ with $S^h \simeq B$ such that $S$ is also $A$-flat.
\end{lemma}
\begin{proof}

By Lemma~\ref{lem:hfpishoffp} there exists a finitely presented $A$-algebra $S$ such that $S^h$ the Henselization of $S$ by $IS$ is isomorphic to $B$ over $A$. To show that we can take $S$ to be flat over $A$, we will find an \Et $S$-algebra which is $A$-flat and also has Henselization $B$.

Let $\psi: S \to B$ be the canonical map from $S$ to its Henselization $B \simeq S^h$. Consider any prime ideal $\mf{p}$ of $S$ which contains $IS$. Since $B \simeq S^h$, we have $B/IB \simeq S/IS$. The prime ideals of $S$ containing $IS$ are in bijective correspondence with the prime ideals of $S/IS \simeq B/IB$, and these are in bijective correspondence with prime ideals of $B$ containing $IB$. This gives us a prime $\mf{P} \subset B$ with $\psi^{-1}(\mf{P})=\mf{p}$, and $\mf{P} \supset IB$.

Now let $f$ be the structure map $f: A \to S$ making $S$ an $A$-algebra. Then we get a prime ideal $\mf{q} = f^{-1}(\mf{p}) \subset A$; in fact $\mf{q}=(\psi f)^{-1}(\mf{P}) = \phi^{-1}(\mf{P})$ since $\psi f = \phi$.

The map $\psi$ from $S$ to its Henselization $B$ is flat. By our hypothesis $\phi=\psi f $ is also flat. Therefore the homomorphisms of local rings $A_{\mf{q}} \to B_{\mf{P}}$ (induced by $\phi$) and $S_{\mf{p}} \to B_{\mf{P}}$ (induced by $\psi$) are flat and in fact faithfully flat by \spcite{00HR}{Lemma}. Therefore the map $A_{\mf{q}} \to S_{\mf{p}}$ is flat by \spcite{039V}{Lemma}. Hence $A \to S$ is flat at all prime ideals in $V(IS)$. 

Therefore we can apply Corollary~\ref{cor:niceprops} to find some flat and finitely presented $A$-algebra $C_1$  which is \Et over $S$ such that $S/IS \simeq C_1/IC_1$. Furthermore $C_1^h \simeq B$ by Proposition~\ref{prop:hensetexp}, so we can replace $S$ by $C_1$, completing the proof.\end{proof}

\begin{proposition}\label{prop:flathenssm} Let $\phi: (A,I) \to (B,IB)$ be a map of Henselian pairs. Then $\phi$ is \hfp, \hfsmooth, and flat if and only if it is \hsmooth.
\end{proposition}
\begin{proof} The strategy of proof below could be used to prove Proposition~\ref{prop:flathenset}, but instead we used a simpler method in that case, avoiding the use of Lemma~\ref{lem:flathfp}.

First assume that $\phi$ is \hsmooth. Then by Proposition~\ref{prop:hetisformal}, $\phi$ is \hfsmooth. By the definition of \hsmooth, in fact $B \simeq R^h$ for $R$ smooth over $A$. Since $R \to R^h \simeq B$ is flat and smooth maps are flat, we also know that $\phi$ is flat. Therefore if $\phi$ is \hsmooth, it is \hfsmooth and flat.

To show the other direction, we assume that $\phi$ is flat, \hfp, and \hfsmooth. By Lemma~\ref{lem:flathfp} we can find a flat and finitely presented $A$-algebra $S$ with $S^h \simeq B$.

Since $B \simeq S^h$, for all $r \ge 1$ we have $S/I^rS \simeq B/I^rB$. By Lemma~\ref{lem:finet}, we see that $S/I^rS$ is smooth over $A/I^r$ for all $r \ge 1$.

We can write $B \simeq S^h$ as the colimit $\colim_C C$ of the \Et $S$-algebras $C$ with $C/IC \simeq S/IS$. Since $S$ is flat and finitely presented over $A$, each $C$ is flat and finitely presented over $A$ as well.

Now fix a prime $\mf{q} \in V(IC) \subset \spec(C)$, which lies over $\mf{p} \supset I$ a prime ideal of $A$. Then the fiber ring $C \tens_A \kappa({\mf{p}})$ is the same as the fiber ring of the map $A/I \to C/IC$ at the prime $\mf{p}/I$. In other words, $\kappa(\mf{p} \in \spec(C)) = \kappa(\mf{p}/I \in \spec(C/I))$ and $C \tens_A \kappa({\mf{p}}) \simeq C/IC \tens_{A/I} \kappa(\mf{p}/I)$ as rings over $\kappa(\mf{p})$. 

The map $A/I \to S/IS \simeq C/IC$ is smooth. Then $C \tens_A \kappa({\mf{p}}) \simeq C/IC \tens_{A/I} \kappa(\mf{p}/I)$ is smooth over $\kappa(\mf{p})$. Then since $A \to C$ is finitely presented and flat, $A \to C$ is smooth at $\mf{q}$ for each prime $\mf{q}$ of $C$ containing $IC$ by \spcite{00TF}{Lemma}.

Because $A \to S$ is smooth at each prime of $V(IS)$, we can apply Corollary~\ref{cor:niceprops} to find some $C_1$ which is \Et over $S$, with $S/IS \simeq C_1/IC_1$, so that $C_1$ is smooth over $A$ (hence finitely presented). Furthermore $C_1^h \simeq S^h \simeq B$ by Proposition~\ref{prop:hensetexp}, so $B$ is \hsmooth over $A$ as we desired to show.\end{proof}
\begin{remark}\label{rmk:kprwrong}

It is ``shown'' in \cite[Satz 3.6.3]{kpr} (which is cited in \cite[Proposition 1]{cox1}) that a map of Henselian schemes $X \to Y$ which is \hfp and \hfetale can be locally described as an open subspace of a ``standard \Et'' map $\sph(((C\{T\}/(g))_{g'})^h) \to \sph(C)$, and hence is flat.  But the local quotient map modulo a nonzero ideal as in Remark~\ref{rmk:quothet} shows that such flatness can fail (beyond the Noetherian case).  Let's explain where the error occurs in the proof of \cite[Satz 3.6.3]{kpr}.

In that proof, it is first shown that if $X \to Y$ is \hfp and \hfetale, it can be realized locally as a closed subspace of such a ``standard \hetale'' map. For $x \in X$ with image $y \in Y$, the realization of $X$ as a locally closed subspace of a ``standard \hetale'' map yields a surjection $((\O_{Y,y}\{T\}/(g))_{g'})^h \to \O_{X,x}$ with finitely generated kernel $N$. Letting $A'=((\O_{Y,y}\{T\}/(g))_{g'})^h/N^2,J=N/N^2$, we have a short exact sequence $$0 \to J \to A' \to \O_{X,x} \to 0,$$ and there exists a unique section $\O_{X,x} \to A'$ because $\O_{Y,y} \to \O_{X,x}$ is \hfetale by assumption.

For $k$ the residue field of the local ring $\O_{Y,y}$, it is shown in \cite{kpr} that the quotient map $$A' \tens_{\O_{Y,y}} k \to \O_{X,x} \tens_{\O_{Y,y}} k$$ is an isomorphism. If $J \tens_{\O_{Y,y}} k$ were the kernel of this map, then it would vanish, so we could apply Nakayama's Lemma (as $J$ is finitely generated) to deduce that $J=0$, hence $N=0$ and $\O_{X,x} = A'$. Then $\O_{X,x}$ would be $\O_{Y,y}$-flat by inspection; that is, the desired flatness of $X \to Y$ at $x$ would follow. However, in general there is no reason for $(\cdot) \tens_{\O_{Y,y}} k$ applied to a short exact sequence to carry the kernel to the kernel, unless the rightmost term of the short exact sequence is known to be flat -- but that would be exactly the $\O_{Y,y}$-flatness of $\O_{X,x}$ that one is aiming to prove.

In Remark~\ref{rmk:quothet} we have an example of a quotient map of Henselian pairs $(A,I) \to (A/J,I/J)$ which is \hfetale but not flat since $J \subseteq \bigcap_{r \ge 1} I^r$ but $J^2 \ne J$. Therefore we know that flatness is not an automatic property of locally \hfp and \hfetale maps (beyond the Noetherian case). Thus the above error in the proof of \cite[Satz 3.6.3]{kpr} cannot be fixed.
\end{remark}

\subsection{The small \hetale site}\label{sub:hetsite}

In forthcoming work \cite{ghga} concerning a Henselian version of formal GAGA, it will be very useful to consider the \hetale topology of a Henselian scheme; thus in this section we define and study the small \hetale site. In particular, we will show that the category of \hetale maps over a Henselian scheme $X$ is equivalent to the category of \Et maps over the underlying ordinary scheme $X_0$ from Remark~\ref{rmk:getususch}. This was discussed in \cite[Theorem 1]{cox1}, relying on the erroneous result in Remark~\ref{rmk:kprwrong}, so our treatment is necessarily different.

We will first show that the category of \hetale algebras over a Henselian pair $(A,I)$ is equivalent to the category of \Et algebras over the quotient $A/I$. 

\begin{proposition}\label{prop:affcatequiv} Let $(A,I)$ be a Henselian pair. The functor $\mf{F}$ which sends an \hetale map $(A,I) \to (B,IB)$ to the \Et map $A/I \to B/IB$ is an equivalence of categories between the category of \hetale algebras over $A$ and the category of \Et algebras over $A/I$.
\end{proposition}
If the ideal $I$ is radical, then the quotient $B/IB$ is reduced since it is \Et over the reduced ring $A/I$. Therefore the ideal $IB$ is also radical.

\begin{proof} We first note that for any \hetale map $(A,I) \to (B,IB)$, the map $A/I \to B/IB$ is  \Et by Proposition~\ref{prop:hetisformal} and Lemma~\ref{lem:finet}, so this functor is well-defined. 

To show that $\mf{F}$ is  an equivalence of categories, we must show that it is both essentially surjective and fully faithful. That is, we must show that for every \Et $A/I$-algebra $\ol{S}$ there exists an \hetale $A$-algebra $S$ so that $S/IS \simeq \ol{S}$, and that for any \hetale $A$-algebras $B,B'$, the natural map $$\Hom_A(B,B') \to \Hom_{A/I}(B/IB,B'/IB')$$ is a bijection.

We begin by showing essential surjectivity. Let $\ol{S}$ be any \Et algebra over $A/I$. We can find an \Et $A$-algebra $S_1$ with $S_1/IS_1 \simeq \ol{S}$ by \spcite{04D1}{Lemma}. Now letting $S = S_1^h$ be the Henselization of the pair $(S_1,IS_1)$, we have that $S/IS \simeq S_1/IS_1 \simeq \ol{S}$ and $S$ is \hetale over $A$. Therefore $\mf{F}$ is essentially surjective. 

We now show fullness of $\mf{F}$ -- that is, given a map $\ol{g}: B/IB \to B'/IB'$ over $A/I$, where $B,B'$ are \hetale $A$-algebras, we wish to find a lift $g: B \to B'$. 

Assume that $R,R'$ are \Et $A$-algebras such that $R^h \simeq B, (R')^h \simeq B'$. Then we can find an \Et $R'$-algebra $R''$ with $R'/IR' \simeq R''/IR''$ and a map $\widetilde{g}: R \to R''$ so that the map $R/IR \to R''/IR''$ induced by $\widetilde{g}$ is equal to $\ol{g}$ via the isomorphisms $R/IR \simeq B/IB, R'/IR' \simeq R''/IR'' \simeq B'/IB'$ \cite[Corollary 2.1.3]{arabia}. 

We see that $(R')^h \simeq (R'')^h \simeq B'$ (Proposition~\ref{prop:hensetexp}). Therefore  the composition of $\widetilde{g}$ with the map $R'' \to (R'')^h \simeq B'$ gives us a map $R \to B'$. This map factors through $R^h \simeq B$ by the universal property of Henselization, so we get a map $g: B \to B'$  which  induces $\ol{g}$ modulo $I$ by construction. Hence $\mf{F}$ is full.

Finally we show faithfulness of the functor; we must show that for $B,B'$ a pair of \hetale $A$-algebras and two maps $f,g \in \Hom_A(B,B')$, if $f$ and $g$ are equal as elements of $\Hom_{A/I}(B/IB,B'/IB')$, then  $f=g$. 

We can once again choose \Et $A$-algebras $R,R'$ such that $R^h \simeq B, (R')^h \simeq B'$. Then $R$ is necessarily finitely presented over $A$. Since the Henselization is defined as a filtered colimit, we can find some \Et $R'$-algebra $R''$ such that $R''/IR'' \simeq R'/IR'$ with maps $\widetilde{f},\widetilde{g}: R \to R''$ which induce $f,g$ respectively on the Henselizations $R^h \to (R'')^h$ (i.e. $ B \to B'$) by \spcite{00QO}{Lemma}.  For ease of notation we can replace $R'$ with $R''$, as $(R'')^h \simeq (R')^h$ (Proposition~\ref{prop:hensetexp}), so we can assume that $\widetilde{f},\widetilde{g}$ are maps $R \to R'$, as in the commutative diagram below.\begin{center}
\begin{tikzcd}
	&[-36pt]R & {R'} &[-36pt] \\
	B \simeq &R^h & (R')^h &\simeq B'
	\arrow["{\widetilde{f}}", shift left, from=1-2, to=1-3]
	\arrow["{\widetilde{g}}"', shift right, from=1-2, to=1-3]
	\arrow[from=1-2, to=2-2]
	\arrow[from=1-3, to=2-3]
	\arrow["f", shift left, from=2-2, to=2-3]
	\arrow["g"', shift right, from=2-2, to=2-3]
\end{tikzcd}
\end{center}Consider the map $\Phi: \spec(R') \to \spec(R \tens_A R)$ defined by the map of rings $\widetilde{f} \tens \widetilde{g}: R \tens_A R \to R'$ where $(\widetilde{f} \tens \widetilde{g})(r \tens s) = \widetilde{f}(r)\widetilde{g}(s)$. 

Because $R$ is \Et over $A$, the image $\Delta$ of the diagonal morphism $\spec(R) \to \spec(R \tens_A R)$ is open. Therefore $E_R=\Phi^{-1}(\Delta) \subset \spec(R')$ is also open, and it must contain $V(IR')$ since $\widetilde{f}$ and $\widetilde{g}$ are equal modulo $I$. 

Let $E_{B'}$ be the inverse image of $E_R$ along the morphism $\spec(B') \to \spec(R')$, which must be open. We then see that $E_{B'}$ contains $V(IB')$. However, since $IB'\subseteq \Jac(B)$, the only open set containing $V(IB')$ is the whole space $\spec(B')$. Therefore $E_{B'}=\spec(B')$.

Because $E_{B'}=\spec(B')$, we see that the map $R \to R^h \simeq B$ equalizes the maps $f,g: B \to B' \simeq (R')^h$. Then by the universal property of Henselization we have $f=g$ as we desired to show. Hence $\mf{F}$ is fully faithful.


Then $\mf{F}$ is essentially surjective and fully faithful, hence an equivalence of categories as we desired to show. \end{proof}

We now define a global analogue of the affine notions of \hetale and \hsmooth maps.

\begin{definition}\label{def:hsmetglob} A morphism $f: (X,\O_X) \to (Y,\O_Y)$ of Henselian schemes is {\bf \hetale} if it can be locally described as a morphism of affine Henselian schemes $\sph(B,J) \to \sph(A,I)$ such that the corresponding map of Henselian pairs $(A,I) \to (B,J)$ is \hetale (so $\sqrt{J}=\sqrt{IB}$). 

Similarly, a morphism $f: (X,\O_X) \to (Y,\O_Y)$ of Henselian schemes is {\bf \hsmooth} if it can be locally described as a morphism of affine Henselian schemes $\sph(B,J) \to \sph(A,I)$ such that the corresponding map of Henselian pairs $(A,I) \to (B,J)$ is \hsmooth (so $\sqrt{J}=\sqrt{IB}$). 
\end{definition}

It is clear that \hetale and \hsmooth morphisms of Henselian schemes are \hlfp. Also, by the Zariski-local nature of the properties of being \hetale or \hsmooth for maps of Henselian pairs (by Lemmas~\ref{lem:finet}~and~\ref{lem:finsm} and Propositions~\ref{prop:flathenset}~and~\ref{prop:flathenssm}) we see that Definition~\ref{def:hsmetglob} recovers Definition~\ref{def:hsmet} for affine Henselian schemes $X$ and $Y$.

We will now prove a global analogue of Proposition~\ref{prop:affcatequiv}. 

\begin{proposition}\label{prop:globcatequiv} Let $(X,\O_X)$ be a Henselian scheme and let $(X_0,\O_{X_0})$ be the scheme defined as in Remark~\ref{rmk:getususch}. The functor $\mf{F}$ from the category of Henselian schemes \hetale over $X$ to the category of schemes \Et over $X_0$ which sends $(Y,\O_Y)$ to $(Y_0,\O_{Y_0})$ (defined as in Remark~\ref{rmk:getususch}) is an equivalence of categories.
\end{proposition}
\begin{proof} To see that $\mf{F}$ is well-defined, we recall the construction in Remark~\ref{rmk:getususch} and see that, locally on Henselian affine opens in $X$, $\mf{F}$ behaves exactly as the functor of Proposition~\ref{prop:affcatequiv}.

We first prove that $\mf{F}$ is faithful. We consider two Henselian schemes $Y,Y'$ both \hetale over $X$ with morphisms $f,g: Y \to Y'$. Assume that the morphisms $f_0,g_0: Y_0 \to Y_0'$ between the corresponding \Et $X_0$-schemes agree. To show $f=g$, we can check Zariski-locally, so we can assume $X$ is Henselian affine. Since $|Y|=|Y_0|$, we can also assume $Y$ is affine and deduce that $f=g$ from Proposition~\ref{prop:affcatequiv}. Thus $\mf{F}$ is faithful.

To see that $\mf{F}$ is full, we again consider two Henselian schemes $Y,Y'$ (\hetale over $X$) with a morphism $f_0: Y_0 \to Y_0'$. Because we have already shown $\mf{F}$ is faithful and $|Y|=|Y_0|,|Y'|=|Y_0|$, we may construct the desired morphism $f: Y \to Y'$ by working Zariski-locally on $Y$. Thus we can again reduce to the affine case and apply Proposition~\ref{prop:affcatequiv} to show that $\mf{F}$ is full.

Finally, we wish to show that $\mf{F}$ is essentially surjective -- that is, given an \Et $X_0$-scheme $Z$ we wish to obtain $Y$ \hetale over $X$ such that $Y_0=Z$. Since $\mf{F}$ is fully faithful, we may construct $Y$ by working Zariski-locally on $X$. Thus we have once more reduced to the affine case Proposition~\ref{prop:affcatequiv}. 

Therefore $\mf{F}$ is essentially surjective, hence an equivalence of categories as we desired to show.
\end{proof}

We will now define the small \hetale site and show that quasi-coherent sheaves for the Zariski topology of a Henselian scheme are sheaves for the \hetale topology.

\begin{definition}\label{def:hetsite}
The {\bf small \hetale site} $X_{\het}$ of a Henselian scheme $X$ is the category of Henselian schemes \hetale over $X$, with the evident notion of a covering. \end{definition}

The functor of Proposition~\ref{prop:globcatequiv} gives an isomorphism of sites $(X_0)_{\et} \to X_{\het}$ for $X_0$ as in Remark~\ref{rmk:getususch}.

Quasi-coherent sheaves on a scheme are sheaves for the \Et topology. We now show that similarly, quasi-coherent sheaves on Henselian schemes are also sheaves on the small \hetale site, as mentioned in Section~\ref{sec:intro}. This will be essential to proving a version of formal GAGA for Henselian schemes in \cite{ghga}, as discussed in Remark~\ref{rmk:ghgahet}.

\thmqcohhet

\begin{proof} It suffices to check that $\F_{\het}$ is an \hetale sheaf Zariski-locally on $X$, since it is a sheaf for the Zariski topology on each object of $X_\het$. Thus we can reduce to the case where $X$ is an affine Henselian scheme $\sph(A,I)$ for a Henselian pair $(A,I)$ and $\F=\widetilde{M}$ for some $A$-module $M$. 

Let $Y=\spec(A)$. For the underlying scheme $X_0$ of $X$, as defined in Remark~\ref{rmk:getususch}, the small Zariski sites $(X_0)_{\zar}$ and $X_{\zar}$ are the same since $X$ and $X_0$ have the same underlying topological space. Furthermore, the small \Et site $(X_0)_{\et}$ and the small \hetale site $X_{\het}$ are equivalent via the functor of Proposition~\ref{prop:globcatequiv}. In fact $X_0=\spec(A/I)$ and we have a canonical closed immersion $\iota: X_0 \to Y$ such that $X$ is the Henselization of $Y$ along the closed subscheme $X_0$. 

 We will show that the presheaf $\F_\het$ is a \hetale sheaf by giving an isomorphism to a sheaf on the equivalent site $(X_0)_\et$. 

Let $\F'$  be the sheaf on $\spec(A)$ associated to $M$. By Lemma~\ref{lem:modpullback}, $\F$ is the pullback of $\F'$ along the canonical map $\sph(A) \to \spec(A)$. We may now define a sheaf $\F_0 = \iota_{\et}^{-1}((\F')_{\et})$ on $(X_0)_{\et}$ (hence a sheaf on $X_\het$). 

Applying Lemma~\ref{lem:tildefunctor} we obtain a map of Zariski sheaves $\F \to \F_0|_{(X_0)_{\zar}}$, which is an isomorphism by Lemma~\ref{lem:fancypullback}. For each affine object $X' =\sph(B,J) \to X$ of $X_\het$ with underlying scheme $X_0'=\spec(B/J) \subseteq \spec(B)=:Y'$, we can apply Lemmas~\ref{lem:tildefunctor}~and~\ref{lem:fancypullback} again to get an isomorphism $\F_\het|_{X'_\zar} \simto \F_0|_{(X_0')_\zar}$ on $X'_\zar=(X_0')_\zar$.

Thus we obtain via ``gluing'' a map $\F_\het \to \F_0$ of presheaves on $X_\het$. To check that this map is an isomorphism, it suffices to work Zariski-locally on each object $Z \to X$ of $X_\het$. Therefore it suffices to compare sections on affine objects $X' \to X$ of $X_\het$. 
Since we have an isomorphism $\F_\het|_{X'_\zar} \simto \F_0|_{(X_0')_\zar}$ given by Lemma~\ref{lem:fancypullback}, in fact $\Gamma(X', \F_{\het}) \simto \Gamma(X_0',\F_0)$ as well. 

Hence $\F_\het \simto \F_0$ is an isomorphism, so $\F_\het$ is a \hetale sheaf as we desired to show.
 \end{proof}

\begin{remark}\label{rmk:morhet} For a quasi-coherent sheaf $\F$ on a scheme $X$ with closed subscheme $Y$ (and Henselization $X^h$ along $Y$), the sheaf $(\F^h)_{\het}$ on the small \hetale site $(X^h)_{\het}$ of $X^h$ is the pullback of $\F_{\et}$ along the morphism of sites $(X^h)_{\het} \to X_{\et}$ corresponding to the morphism of locally ringed spaces $X^h \to X$. (This follows from Lemma~\ref{lem:fancypullback} and the construction of $\F_0$ in the proof of Theorem~\ref{thm:qcohhet}.) Thus to simplify notation, we will often write $(\F_{\et})^h$ for $(\F^h)_{\het}$.

It is also clear that pullback and ``finite pushforward'' (see Lemma~\ref{lem:modpush}) of quasi-coherent sheaves along morphisms of Henselian schemes are compatible with the functor $\F \mapsto \F_{\het}$ of Theorem~\ref{thm:qcohhet}. 

\end{remark}

\begin{lemma}\label{lem:stalkhet} Let $X$ be a Henselian scheme, with $X_0$ the scheme with the same underlying topological space as $X$ defined in Remark~\ref{rmk:getususch}. For $\ol{x}_0$ a geometric point of $X_0$ lying over $x \in X_0$, consider $x$ as a point of $X$ and write $\ol{x}$ for the geometric point of $X$ arising from $\ol{x}_0$. Then the following statements hold:

\begin{enumerate}[(i)]
\item the stalk of the structure sheaf $\O_{X_{\het}}$ of the small \hetale site at $\ol{x}$ is isomorphic to $\O_{X,x}^{\rm sh}$, the strict Henselization of the local ring $\O_{X,x}$ along its {\it maximal ideal},

\item if $X$ is the Henselization of a scheme $Y$ along a closed subscheme $Z$, then considering $\ol{x}_0$ as a geometric point of $Z=X_0 \subset Y$, the stalk of $\O_{Y_{\et}}$ at $\ol{x}_0$ is isomorphic to the stalk of $\O_{X_{\het}}$ at $\ol{x}$. 
\end{enumerate}
\end{lemma}
\begin{proof}

We will first prove the statement (ii). 

In the situation of (ii), the small \Et site $Z_{\et}$ and the small \hetale site $X_{\het}$ are equivalent by Proposition~\ref{prop:globcatequiv}. As in the proof of Theorem~\ref{thm:qcohhet}, we see that  $\O_{X_{\het}}=\iota_{\et}^{-1}(\O_{Y_{\et}})$ as sheaves on $Z_{\et}$. Then by the definition of the inverse image functor, we get the desired equality of stalks \spcite{03Q1}{Lemma}.

To prove (i), we note that the stalk of $\O_{X_{\het}}$ at $\ol{x}$ can be computed Henselian affine-locally on $X$; then since a Henselian affine scheme $\sph(A,I)$ is the Henselization of $\spec(A)$ along $\spec(A/I)$, we can reduce to the situation of (ii). Hence we assume that $X$ is the Henselization of a scheme $Y$ along a closed subscheme $Z$, with $\ol{x}_0$ a geometric point of $Z$ over a point $x \in Z \subset Y$.

By (ii), we have an isomorphism of stalks $\O_{Y_{\et},\ol{x}_0} \simeq \O_{X_{\het},\ol{x}}$. By \spcite{04HX}{Lemma} we have $\O_{Y_{\et},\ol{x}_0} \simeq (\O_{Y,x})^{\rm sh}$, where the strict Henselization is taken with respect to the {\it maximal ideal} of $\O_{Y,x}$.

If we let $\mc{I}$ be the radical ideal sheaf corresponding to the closed subscheme $Z \subset Y$, we know that $\O_{X,x}$ is the $\mc{I}_x$-adic Henselization of $\O_{Y,x}$. These local rings have the same residue field.

Since the strict Henselization  $(\O_{Y,x})^{\rm sh}$ is Henselian for the maximal ideal of $\O_{Y,x}$, which contains $\mc{I}_x$, by \spcite{0DYD}{Lemma} the ring $(\O_{Y,x})^{\rm sh}$ is also Henselian with respect to $\mc{I}_x(\O_{Y,x})^{\rm sh}$. Therefore the morphism $\O_{Y,x} \to (\O_{Y,x})^{\rm sh}$ factors through $\O_{X,x}$. Both $\O_{X,x}$ and $(\O_{Y,x})^{\rm sh}$ are filtered colimits of \Et $\O_{Y,x}$-algebras, so $\O_{X,x} \to (\O_{Y,x})^{\rm sh}$ is a filtered colimit of \Et ring maps by \spcite{08HS}{Lemma}.
 
 Then by uniqueness properties of the strict Henselization \spcite{08HT}{Lemma}, we see that the map $(\O_{X,x})^{\rm sh} \to (\O_{Y,x})^{\rm sh}$ is an isomorphism, thus proving (i).\end{proof}

We will now compare the cohomology of quasi-coherent sheaves on the small \hetale site to Zariski cohomology. To simplify notation, for a quasi-coherent sheaf $\F$ on a Henselian scheme $X$, we write $\H^i_{\het}(X,\F)$ for the cohomology group $\H^i(X_{\het},\F_{\het})$. Similarly for a quasi-coherent sheaf $\G$ on a scheme $Y$, we write $\H^i_{\et}(Y,\G)$ for $\H^i(Y_{\et},\G_{\et})$. We will first show that the \hetale cohomology of a quasi-coherent sheaf on a Henselian affine scheme of characteristic $p > 0$ vanishes in positive degrees.

\begin{lemma}\label{lem:thmbhet} Let $(A,I)$ be a Henselian pair such that $A$ has characteristic $p > 0$. Then for a quasi-coherent sheaf $\F$ on $\spec(A)$ and any affine object $Z$ of $\sph(A)_{\het}$, the \hetale cohomologies $\H^j_{\het}(Z,\F^h)$ are $0$ for $j > 0$.
\end{lemma}

The counterexample given by de Jong in \cite{nothmbblog} to the analogous vanishing for Zariski cohomology for $A$ of characteristic $0$ (described here in Proposition~\ref{prop:cxyhglobsec}) also works in the \hetale topology. Therefore in Lemma~\ref{lem:thmbhet} we cannot drop the assumption that $A$ has positive characteristic. 

\begin{proof}[Proof of Lemma~\ref{lem:thmbhet}] The same fact for Zariski cohomology was proved by de Jong in \cite{thmbblog}, which here is Theorem~\ref{thm:thmbpos}. The proof relies on \spcite{09ZI}{Theorem}, a result of Gabber concerning torsion abelian \Et sheaves, which can only be used in positive characteristic as in that case all abelian sheaves are torsion.

Let $Y=\spec(A/I)$ and $X=\spec(A)$, with $X^h=\sph(A)$. Then $Y$ and $X^h$ have the same underlying topological space.

 If we let $i: Y \hookrightarrow X$ be the canonical closed immersion, the sheaf $(i_{\et}^{-1}(\F_{\et}))|_{Y_{\zar}}$ on $Y$ is a sheaf of modules over $i_{\et}^{-1}(\O_{X_{\et}})|_{Y_{\zar}} = \O_{X^h}$, and is equal to the pullback $\F^h$ of $\F$ to $X^h$ (Lemma~\ref{lem:fancypullback}). As a refinement, the sheaf $i_{\et}^{-1}(\F_{\et})$ on $Y_{\et}$, via the isomorphism of sites $Y_{\et} \to (X^h)_{\het}$ given by Proposition~\ref{prop:globcatequiv}, is the sheaf $(\F^h)_{\het}$ by Remark~\ref{rmk:morhet}.

Let $Z=\sph(B^h)$ be an affine object of $(X^h)_{\het}$, and $Z_0=\spec(B/IB)=\spec(B^h/IB^h)$ the corresponding affine object of $Y_{\et}$ (where $B$ is an \Et $A$-algebra). 

Since $(B^h,IB^h)$ is a Henselian pair and $\F,\F^h$ are both necessarily $p$-torsion, by \spcite{09ZI}{Theorem} we have $\H^j_{\et}(\spec(B^h),\F) = \H^j_{\et}(\spec(B^h/IB^h),i_{\et}^{-1}(\F))$ for all $j$. Then for all $j>0$ we have $$\H^j_{\et}(\spec(B^h),\F) = \H^j_{\et}(Z_0,i_{\et}^{-1}(\F_{\et}))=\H^j_\het(Z,\F^h)=0,$$ completing the proof.
\end{proof}

Now, as mentioned in Section~\ref{sec:intro}, we can prove the desired comparison between Zariski cohomology and \hetale cohomology on Henselian schemes. This comparison is only an isomorphism in positive characteristic since it relies on Lemma~\ref{lem:thmbhet},  and this isomorphism result is vital to deriving a Henselian version of formal GAGA in positive characteristic in forthcoming work \cite{ghga} (see Remark~\ref{rmk:ghgahet}).

\thmhetzarcoh
\begin{proof} We have a continuous map of sites $\pi: X_{\het} \to X_{\zar}$, with the associated pushforward functor of sheaves defined by $(\pi_*\F)(U \;\subset\; X) = \F(U \to X)$.

The left adjoint $\pi^{-1}$ of $\pi_*$, described in Theorem~\ref{thm:qcohhet}, is an exact functor. Hence the pushforward of an injective abelian sheaf is also injective.

Now let $\G$ be a quasi-coherent sheaf of $\O_X$-modules, with $\F=\G_{\het}$; then $\pi_*\F=\G$. In the Leray spectral sequence $E_2^{i,j}= \H^i(X,R^j\pi_*\F) \implies \H^{i+j}(X_{\het},\F) = \H^{i+j}_{\het}(X,\G)$ \spcite{0732}{Lemma}, the edge map $\H^i(X,\G) = H^i(X,\pi_*\F) \to H^i_{\het}(X_{\het},\F)$ is the natural map. Thus, if we can show that $R^j\pi_*\F=0$ for all $j > 0$, we have the desired isomorphism \spcite{0733}{Lemma}.

The higher direct image $R^j\pi_*\F$ is the sheafification of the presheaf $V \mapsto \H^j_\het(V,\F)=\H^j_\het(V,\G)$ on $X_\zar$. We may thus compute the stalk at a point $x \in X$ as the colimit $\colim_{x \in U} \H^j_{\het}(U,\G)$ taken over open neighborhoods $U$ of $x$ in $X$. In fact, 
we can take the colimit over the system of Henselian affine open neighborhoods of $x$, since this system is cofinal.

By Lemma~\ref{lem:thmbhet} for $j > 0$ we have $\H^j_{\het}(U,\G)=0$ for $U$ Henselian affine; thus $$(R^j\pi_*\F)_x=\colim_{\substack{x \in U \\ U \text{ Hens. affine}}} \H^j_{\het}(U,\G) =0.$$ Hence $R^j\pi_*\F=0$ for $j>0$ as we desired to show. \end{proof}


\begin{definition}\label{def:hetcmpsnmap}For a scheme $X$ with Henselization $X^h$ along a closed subscheme $Y$ and a quasi-coherent sheaf $\F$ on $X$, the equality $(\F^h)_{\het} = i_{\et}^{-1}(\F_{\et})$ for $i: Y \hookrightarrow X$ the canonical closed immersion gives rise to a {\bf \hetale cohomology comparison map} $\H^j_{\et}(X,\F) \to \H^j_{\het}(X^h,\F^h)$.
\end{definition}

\corhetzarcmpsn
\begin{proof} 

We have a commutative diagram 
\begin{center}
\begin{tikzcd}
\H^j(X,\F) \arrow[r]\arrow[d,Isom]& \H^j(X^h,\F^h)\arrow[d,Isom]\\
\H^j_{\et}(X,\F) \arrow[r]& \H^j_{\het}(X^h,\F^h)\\
\end{tikzcd}
\end{center} \vspace{-.25in}
where the left vertical arrow is an isomorphism by standard facts about \Et cohomology of quasi-coherent sheaves \spcite{03DW}{Proposition}, and the right vertical arrow is an isomorphism by Theorem~\ref{thm:hetzarcoh}. It is therefore immediate that the upper horizontal arrow is an isomorphism if and only if the lower horizontal arrow is an isomorphism.\end{proof}

\section{Appendix: Henselian rings}\label{sec:appendix}

In this appendix we collect certain useful facts and proofs about Henselian pairs and Henselization as a convenient reference for the reader.

Some equivalent definitions of a Henselian pair are given in \spcite{09XI}{Lemma}, which we state here without proof. 

\begin{lemma}\label{lem:sphenspair} Let $A$ be a ring and $I \subsetneq A$ an ideal. 

The following are equivalent \begin{enumerate}[(a)]

\item $(A,I)$ is a Henselian pair,

\item Given a diagram 

\begin{center}
 \begin{tikzcd} A \arrow[r,"\text{\Et}",shift left=.5ex]\arrow[d] & A' \arrow[l,dashed,shift left=.5ex] \arrow[d]\\
A/I \arrow[r,shift left=.5ex] & \arrow[l,"\sigma",shift left=.5ex] A'/IA'\\
\end{tikzcd}
\end{center} \vspace{-.25in}

for an \Et ring map $A \to A'$, the corresponding map $A/I \to A'/IA'$, and a section $\sigma: A'/IA' \to A/I$, there exists a section $A' \to A$ (the dashed arrow) lifting $\sigma$,

\item for any finite $A$-algebra $B$ the map $B \to B/IB$ induces a bijection on idempotents,

\item for any integral $A$-algebra $B$ the map $B \to B/IB$ induces a bijection on idempotents,

\item (Gabber) $I$ is contained in the Jacobson radical $\Jac(A)$ and every monic polynomial $f(T) \in A[T]$ of the form

\begin{equation*} f(T) = T^n(T-1) + a_nT^n+\dots+a_1T+a_0 \end{equation*} 
with $a_n,\dots,a_0 \in I$ and $n \ge 1$ has a root $\alpha \in 1+I$.
\end{enumerate}

Moreover, in part (e) the root is unique.
\end{lemma}

When we use the condition \ref{lem:sphenspair}(b), often the section $\sigma: A'/IA' \to A/I$ will arise from an $A$-algebra map $\sigma': A' \to A/I$ (so the kernel of $\sigma'$ will necessarily contain $IA'$). We might say that the section $A' \to A$ lifts $\sigma'$ rather than saying that it lifts $\sigma$. 

The very explicit condition \ref{lem:sphenspair}(e) is useful for showing that certain rings are Henselian, as in the following example. 

\begin{example}\label{ex:hensnotsep}

Let $R=\mc{C}_0^\infty(\R^n,\R)$ be the ring of germs of smooth functions from $\R^n$ to $\R$ at the origin, and let $I \subset R$ be the ideal of functions which are $0$ at the origin. By Hadamard's lemma \cite[\nopp 2.8]{hadcite}, $I$ is finitely generated by the coordinate functions $x_1,\dots,x_n$. It is also clear that $R$ is a local ring with maximal ideal $I$, so $I\subseteq\Jac(R)$.

Now, by the condition Lemma~\ref{lem:sphenspair}(e), to show that $(R,I)$ is a Henselian pair it suffices to show that for any monic polynomial $g(T) \in R[T]$ of the form $$g(T) = T^ n(T - 1) + a_ n T^ n + \ldots + a_1 T + a_0$$ where $a_0,\dots,a_n \in I$ and $n \ge 1$, $g$ has a root $\alpha \in \mathbf{1}+I$.

We fix such a $g$. Note that $R/I \simeq \R$, and the reduction $\ol{g}(T) \in \R[T]$ is equal to $T^n(T-1)$; hence $\ol{g}(1)=0, \ol{g}'(1)=1 \ne 0$. 

Since each of the $a_i$ are elements of $R$, they are germs of smooth functions $\R^n \to \R$ at the origin, which we also denote by $a_i$. Now we can define a function $G: \R^{n+1} \to \R$ by $$G(\mathbf{x},y)=y^n(y-1) + a_ n(\mathbf{x}) y^ n + \ldots + a_1(\mathbf{x}) y + a_0(\mathbf{x})$$ for $\mathbf{x} \in \R^n$ and $y \in \R$. 

Since the $a_i$ are in the ideal $I$, we see that $a_i(\mathbf{0})=0$ for all $i$ (where we write $\mathbf{0}$ for the origin in $\R^n$ for simplicity's sake).  Therefore $G(\mathbf{0},1)=0$. Furthermore, it is clear that $G$ is smooth, and the fact that $\ol{g}'(1)=1 \ne 0$ means we can apply the implicit function theorem to get a neighborhood $U$ of $\mathbf{0}$ and a smooth function $q: U \to \R$ with $q(\mathbf{0})=1$ so that $G(\mathbf{x},q(\mathbf{x}))=0$ for all $\mathbf{x} \in U$. Then, considering $q$ as an element of $R$, it is evident that $q \in \mathbf{1}+I$ and that $q$ is our desired root of $g(T) \in R[T]$. 

Therefore we see that $(R,I)$ is a Henselian pair. In fact $(R,I)$ gives an example of a Henselian pair for which $R$ is not $I$-adically separated, since $e^{-1/||\mathbf{x}||^2} \in \bigcap_{r \ge 1} I^r$. \end{example}

For a local ring $R$, there exists a unique (up to $R$-isomorphism) Henselian local ring $R^h$ with a local ring homomorphism $R \to R^h$ that is a filtered colimit of \Et ring maps. Furthermore, the completion of $R^h$ at its maximal ideal is isomorphic to the completion  $R^\wedge$ of $R$, and the construction of $R^h$ is functorial in $R$. (For an exposition of Henselization of local rings, see \spcite{0BSK}{Section}.) 

Under mild hypotheses on a normal Noetherian local ring $R$, its Henselization is the ``algebraic closure'' of $R$ in $R^\wedge$ -- i.e., $R^h$ consists of exactly the elements of $R^\wedge$ which are solutions of polynomial equations over $R$. (See \spcite{07QX}{Section}, which uses the deep Artin-Popescu Approximation Theorem.)

Given the definition of Henselian pairs above, we can ask for a similar notion of Henselization in the global setting that, in the Noetherian case, serves as an intermediate ``algebraic'' step towards the completion. 

\begin{definition}\label{def:henselization} For any pair $(A,I)$, the {\bf Henselization of $(A,I)$} is the pair $(A^h,I^h)$ for $A^h$ the filtered colimit of the category of \Et ring maps $A \to B$ such that $A/I \to B/IB$ is an isomorphism and $I^h=IA^h$.\footnote{The Henselization of $(A,I)$ is sometimes defined as the colimit of the larger category of \Et ring maps $A \to B$ admitting an $A$-algebra map $B \to A/I$.}
\end{definition}

In fact $(A^h,I^h)$ is a Henselian pair with $A/I \simeq A^h/I^h$. Furthermore, any morphism of pairs $(A,I) \to (A',I')$ with the latter pair being Henselian must factor uniquely through the canonical morphism $(A,I) \to (A^h,I^h)$. This is shown in \spcite{0A02}{Lemma}, and we call it the {\bf universal property of Henselization.}

We also say that $A^h$ is the {\bf Henselization of $A$ along $I$} or the {\bf $I$-adic Henselization of $A$}, and we may omit $I$ if it is clear from context. For a Henselian pair $(A,I)$, we often refer to the Henselization of the pair $(B,IB)$ for an $A$-algebra $B$ as the {\bf $I$-adic Henselization of $B$.}

Since the Henselization of a ring $A$ along an ideal $I$ is an ``intermediate step'' towards the completion of $A$ along $I$, one would expect it to have some of the same nice properties as Noetherian completion. For example, the morphism $A \to A^h$ is flat, and it induces an isomorphism $A/I^r \to A^h/I^rA^h$ for all $r \ge 1$  \spcite{0AGU}{Lemma}. 

\begin{remark}\label{rmk:radideal}
Furthermore, for two ideals $I,J$ of a ring $A$, if $V(I)=V(J)$ -- equivalently, $\sqrt{I}=\sqrt{J}$ -- then the Henselizations of the pairs $(A,I)$ and $(A,J)$ are uniquely isomorphic as $A$-algebras \spcite{0F0L}{Lemma}. More generally, if $V(I)=V(J)$ for two ideals $I,J$ of a ring $A$, then $(A,I)$ is Henselian exactly when $(A,J)$ is \spcite{09XJ}{Lemma}.\end{remark}

From the definition of the Henselization of a pair, we can show that certain \Et algebras have the same Henselization as the base ring:

\begin{proposition}\label{prop:hensetexp} If $(R,I) \to (S,IS)$ is a map of pairs with $R \to S$ \Et and $R/I \simeq S/IS$, then the $I$-adic Henselization $R^h$ of $R$ is isomorphic to the $IS$-adic Henselization $S^h$ of $S$.
\end{proposition}
\begin{proof} 

We recall from Definition~\ref{def:henselization} that $R^h$ is the filtered colimit of the category $\mf{C}_R$ of \Et ring maps $R \to R'$ such that $R/I \to R'/IR'$ is an isomorphism. Similarly, $S^h$ is the filtered colimit of the category $\mf{C}_S$ of \Et ring maps $S \to S'$ such that $S/IS \to S'/IS'$ is an isomorphism. 

By hypothesis, $S \in \mf{C}_R$. We see by basic properties of \Et maps that any $S' \in \mf{C}_S$ is also an object of $\mf{C}_R$, and $R' \tens_R S \in \mf{C}_S$ for any $R' \in \mf{C}_R$. Thus we see that the filtered colimit of $\mf{C}_S$ is canonically isomorphic to the filtered colimit of $\mf{C}_R$.\end{proof}

If $A \to B$ is a module-finite ring homomorphism and $(A,I)$ is Henselian, then $(B,IB)$ is also a Henselian pair \spcite{09XK}{Lemma}. Furthermore, Henselization is commutes with base change by a module-finite ring homomorphism:

\begin{lemma}\label{lem:hensmodfin}If $(R,I) \to (S,J)$ is a module-finite map of pairs with $J=IS$ (such as a surjection $R \twoheadrightarrow S$) we have $R^h \tens_R S \simeq S^h$ for $R^h,S^h$ the $I$-adic  Henselizations of $R$ and $S$ respectively.
\end{lemma}
\begin{proof} This is \spcite{0DYE}{Lemma}. \end{proof}

A finitely presented $A$-algebra is not generally $I$-adically Henselian when $(A,I)$ is. Therefore in order to define a notion of ``\hfp'' in Section~\ref{sec:lfp}:, we need to use the Henselization of polynomial rings.

\begin{definition}\label{def:henspoly}
For a Henselian pair $(A,I)$, the {\bf Henselized polynomial ring} $$A\{X_1,\dots,X_n\}$$ is the Henselization of the pair $(A[X_1,\dots,X_n],IA[X_1,\dots,X_n])$.

The Henselized polynomial ring over $A$ in $n$ variables $B=A\{X_1,\dots,X_n\}$ can be described by the following universal property: for any map of Henselian pairs $(A,I) \to (C,J)$ and $n$ elements $c_1,\dots,c_n \in C$, there exists a unique $A$-algebra map $f: B \to C$ such that $f(X_i)=c_i$ for all $i$. This follows from the universal property of polynomial rings and the universal property of Henselization.
 \end{definition}

\begin{remark}\label{rmk:hfpex}
For $(A,I)$ a Henselian pair and a finitely presented $A$-algebra $$B = A[X_1,\dots,X_n]/(g_1,\dots,g_m),$$ by Lemma~\ref{lem:hensmodfin}, we have $$B^h \simeq A\{X_1,\dots,X_n\} \tens_{A[X_1,\dots,X_n]} B \simeq A\{X_1,\dots,X_n\}/(g_1,\dots,g_m).$$
\end{remark}

Henselized polynomial rings behave like usual polynomial rings with regard to adjoining more variables, as shown in the following lemma:

\begin{lemma}\label{lem:polyonpoly} Let $(A,I)$ be a Henselian pair, and let $B=A\{X_1,\dots,X_n\}$ be the Henselized polynomial ring over $A$ in $n$ variables. Then $B\{Y_1,\dots,Y_m\}= A\{x_1,\dots,x_n,y_1,\dots,y_m\}$ is the Henselized polynomial ring over $A$ in $n+m$ variables. 
\end{lemma}
\begin{proof} This follows from the universal property described after Definition~\ref{def:henspoly}.\end{proof}

When defining \hsmooth and \hetale maps of Henselian rings in Section~\ref{sec:etsm}, it is useful to be able to ``descend'' certain nice properties from the Henselization of a ring to an \Et algebra over it. We formalize this idea in the following lemma and corollary.

\begin{lemma}\label{lem:getbiggeret} Let $\mathbf{P}$ be a property of maps of rings. We say that a morphism $f: A \to A'$ has $\mathbf{P}$ at a point $x' \in \spec(A')$ if $\spec(f)$ restricted to an affine open neighborhood of $x'$ has property $\mathbf{P}$. Assume all of the following:
\begin{enumerate}[(i)]
\item for a finitely presented morphism $f: B \to B'$, the open set $E=\{\mf{p}' \in \spec(B'): \;$f$ \text{ has } \mathbf{P} \text{ at } \mf{p}'\}$ is ind-constructible, and if $E=\spec(B')$, the map $f$ has property $\mathbf{P}$;
\item \Et maps have $\mathbf{P}$;
\item for two morphisms $g: C \to C', h: C' \to C''$, if $h$ has property $\mathbf{P}$ at a prime $\mf{q}$ of $C''$ and $g$ also has property $\mathbf{P}$ at the prime $h^{-1}(\mf{q})$ of $C'$, then the map $hg$ has property $\mathbf{P}$ at $\mf{q}$. 
\end{enumerate}

Let $(A,I)$ be a Henselian pair and $R$ a finitely presented $A$-algebra, and assume that the map $A \to R$ has property $\mathbf{P}$ at each prime in $V(IR) \subset \spec(R)$. Then we can find some $R'$ \Et over $R$ such that $R'/IR' \simeq R/IR$ and the map $A \to R'$ has property $\mathbf{P}$. 
\end{lemma}
\begin{proof} 

We can write the Henselization $R^h$ of the pair $(R,IR)$ as $R = \colim S$ where the colimit is taken over \Et $R$-algebras $S$ such that $R/IR \simeq S/IS$. 

We now define $V_R \subset \spec(R)$ as the set of points at which $\spec(R) \to \spec(A)$ has property $\mathbf{P}$; for each $S$ we define $V_S \subset \spec(S)$ similarly. Any map $g: S \to S'$ in the system of the colimit is necessarily \Et. Then if $\spec(g): \spec(S') \to \spec(S)$ is the map of spectra corresponding to $g$, we can apply (ii) and (iii) to see that $\spec(g)^{-1}(V_S) \subseteq V_{S'}$ and $V(IS) \subset V_S$ for all $S$. By (i), each $V_S$ is ind-constructible and open.

If we define $V \subset \spec(R^h)$ to be the union of the pullbacks of all of the $V_S$'s, then $V$ is also necessarily open, with $V \supset V(IR^h)$ since for all $S$ we have $V_S \supset V(IS)$. 

However since $(R^h,IR^h)$ is Henselian, $IR^h\subseteq \Jac(R^h)$. Therefore $V$, since it is open and contains $V(IR^h)$, must  be equal to all of $\spec(R^h)$.

Hence can find some $R_{\mathbf{P}}$ which is \Et over $R$, with $R/IR \simeq R_{\mathbf{P}}/IR_{\mathbf{P}}$, so that $V_{R_{\mathbf{P}}}=\spec(R_{\mathbf{P}})$, by applying \cite[Corollary 8.3.4]{ega43} to $V_S$ and $V$. Then $R_{\mathbf{P}}$ is the $R'$ we desired to find (again by (i)).\end{proof}

\corniceprops
\begin{proof} 

We first verify the first half of condition (i) of Lemma~\ref{lem:getbiggeret} for the properties of being quasi-finite, flat, or smooth. 

In \cite[Section 6.2]{ega2}, a finite type morphism $f$ is defined to be quasi-finite at a point $x$ if it is isolated in its fiber $f^{-1}(f(x))$. The openness of the locus of such points is shown in \cite[Corollary 13.1.4]{ega43}, and the local constructibility of this locus is shown in \cite[Proposition 9.6.1]{ega43}.

The openness of the flat locus of a finitely presented morphism (i.e., the locus in the source where the map of local rings is flat) is shown in \cite[Theorem 11.3.1]{ega43}, and the local constructibility of this locus is shown in \cite[Proposition 11.2.8]{ega43}.

Finally, the smooth locus of a morphism is defined to be open in \cite[Definition 17.3.7]{ega44}, and this locus is shown to be locally constructible in \cite[Proposition 17.7.11]{ega44} .

It is then not difficult to verify that the second half of condition (i) and conditions (ii) and (iii) of Lemma~\ref{lem:getbiggeret} are true if $\mathbf{P}$ is the property of being quasi-finite, flat, or smooth.\end{proof}

\printbibliography \end{document}